\documentclass{amsproc}

\usepackage{}

\usepackage[normalem]{ulem}
\usepackage{enumitem}
\usepackage{amsmath}
\usepackage{amsfonts}
\usepackage{amssymb}
\usepackage{color}
\usepackage{graphicx,mathtools}
\usepackage{multirow}
\usepackage{float}
\usepackage{soul}

\usepackage{latexsym}

\numberwithin{equation}{section}

\theoremstyle{plain}
\newtheorem{thm}{Theorem}[section]
\newtheorem{lem}[thm]{Lemma}
\newtheorem{prop}[thm]{Proposition}
\newtheorem{cor}[thm]{Corollary}

\theoremstyle{definition}

\newtheorem{deff}[thm]{Definition}

\newtheorem{example}[thm]{Example}
\newtheorem{remark}[thm]{Remark}
\newtheorem{question}[thm]{Question}
\newtheorem{introques}{Question}

\newcommand{\kk}{\Bbbk}

\theoremstyle{remark}

\def\Z{\mathbb{Z}}

\def\MM{\mathcal{M}}
\def\OO{\mathcal{O}}
\def\G{\mathcal{G}}

\def\fg{\mathfrak{g}}
\def\fa{\mathfrak{a}}

\def\s{\sigma}
\def\fm{\mathfrak{m}}
\def\a{\alpha}
\def\g{\gamma}
\def\o{\otimes}
\def\e{\epsilon}
\def\to{\rightarrow}
\def\ol{\overline}
\def\i{\mathbf{i}}

\def\tr{\operatorname{tr}}
\def\ord{\operatorname{ord}}

\def\id{\operatorname{id}}
\def\Ext{\operatorname{Ext}}
\def\End{\operatorname{End}}
\def\Hom{\operatorname{Hom}}
\def\irr{\operatorname{Irr}}
\def\ad{\operatorname{ad}}
\def\stab{\operatorname{stab}}
\def\GL{\operatorname{GL}}
\def\Reg{\operatorname{Reg}}

\def\b{\beta}

\begin{document}
\nocite{*}
\title{On Hopf algebras of dimension $p^n$ in  characteristic $p$}
\date{March 2023}

\author[Ng]{Siu-Hung Ng}
\address{Department of Mathematics, Louisiana State University, Baton Rouge, LA 70803, USA} 
\email{rng@math.lsu.edu}

\author[Wang]{Xingting Wang}
\address{Department of Mathematics, Louisiana State University, Baton Rouge, LA 70803, USA}
\email{xingtingwang@lsu.edu}


\date{}

\maketitle
{\centering\it This paper is dedicated to the memory of Earl J. Taft.\par}

\begin{abstract}
Let $\kk$ be an algebraically closed field of characteristic $p>0$. We study the general structures of $p^n$-dimensional Hopf algebras over $\kk$ with $p^{n-1}$ group-like elements or a primitive element generating a $p^{n-1}$-dimensional Hopf subalgebra. As applications, we have proved that Hopf algebras  of dimension $p^2$ over $\kk$ are pointed or basic for $p \le 5$, and provided a list of characterizations of the Radford algebra $R(p)$. In particular, $R(p)$ is the unique nontrivial extension of $\kk[C_p]^*$ by $\kk[C_p]$, where $C_p$ is the cyclic group of order $p$.  In addition, we have proved a vanishing theorem for some 2nd Sweedler cohomology group and investigated the extensions of $p$-dimensional Hopf algebras. All these extensions have been identified and shown to be pointed.
\end{abstract}

\section{Introduction}
Hopf algebras are generalizations of groups and Lie algebras. In the last few decades, they have played an essential role in invariants of knots, links, and 3-manifolds via their representations. In particular, the representation categories of finite-dimensional Hopf algebras provide fundamental examples of finite tensor categories (see, e.g., \cite{Kas, EGNO}).

Let $\kk$ be an algebraically closed field and $p$ a prime number. It has long been a program classifying finite-dimensional Hopf algebras over $\kk$ of characteristic zero. Kaplansky first conjectured that the only $p$-dimensional Hopf algebra in characteristic zero is the group algebra, later proved by Zhu \cite{Zhu94}. For $p^2$-dimensional Hopf algebras, Masuoka proved that the semisimple ones are group algebras \cite{Masuoka96}. At that time, the only known nonsemisimple Hopf algebras of dimension $p^2$ were the Taft algebras, which were introduced by Taft in \cite{Taft} and are described as follows.

\begin{deff}\cite{Taft}
  Let $\kk$ be a field, and $\omega \in \kk$ a primitive $n$th root of unity for some positive integer $n$. The \emph{Taft algebra} $T_n(\omega)$ is the free algebra $\kk\langle x,g\rangle$ subject to the relations
\[
g^n=1,\quad x^n=0,\quad xg=\omega gx.
\]
The Hopf algebra structure on the Taft algebra $T_p(\omega)$ is given by
\begin{align*}
\Delta(g)&=g\otimes g, & \epsilon(g)&=1, & S(g)&=g^{-1},\\
\Delta(x)&=x\otimes 1+g\otimes x, & \epsilon(x)&=0, &S(x)&=-g^{-1}x.
\end{align*} 
\end{deff}

Montgomery, in several international conferences, asked (cf. Section 5 of \cite{Mont98}):

\begin{question}
   Suppose the base field $\kk$ is algebraically closed of characteristic zero, and $p$ is a prime number. Are the Taft algebras $T_p(\omega)$ the only nonsemisimple Hopf algebras in dimension $p^2$? 
\end{question}
The first author of this paper gave an affirmative answer to Montgomery's question in \cite{Ng02}. 

Meanwhile, the classification of finite-dimensional Hopf algebras over an algebraically closed field $\kk$ of positive characteristic $p$ remains wide open. For any semisimple and cosemisimple Hopf algebra in positive characteristic, Etingof and Gelaki showed that it could be lifted to a semisimple Hopf algebra in characteristic zero \cite{EG98}. Hence it yields the same classification result. Recently, the classification of the $p$-dimensional Hopf algebras over $\kk$ in characteristic $p$ has been accomplished by the authors of this paper in \cite{NW19}. Etingof later generalizes this result to finite tensor categories in \cite{Etingof20}. 

Turning to the $p^2$-dimensional Hopf algebras, the pointed ones are classified in \cite{WW14}. Based on this classification,  there is only one \emph{pointed} noncommutative and noncocommutative $p^2$-dimensional Hopf algebra. This Hopf algebra was discovered by Radford in \cite{Rad1977} and is described as follows.

\begin{deff}\cite{Rad1977}
Let $\kk$ be an algebraically closed field of characteristic $p>0$. The \emph{Radford algebra} $R(p)$ is the free algebra $\kk\langle x,g\rangle$ subject to the relations
\[
g^p=1,\quad x^p=x,\quad [g,x]=g(g-1).
\]
The Hopf algebra structure on $R(p)$ is given by
\begin{align*}
\Delta(g)&=g\otimes g, & \epsilon(g)&=1, & S(g)&=g^{-1},\\
\Delta(x)&=x\otimes 1+g\otimes x, & \epsilon(x)&=0, &S(x)&=-g^{-1}x.
\end{align*}
\end{deff}
It is clear that the Radford algebra could be over any field of characteristic $p$. However, we generally assume $\kk$ is algebraically closed and the notation $R(p)$ is defined under this assumption.

The first author of this paper, on many occasions, asked the following question: 

\begin{introques}\label{Q2}
Let $\kk$ be an algebraically closed field of characteristic $p > 0$.  Is the Radford algebra $R(p)$ the unique noncommutative and noncocommutative $p^2$-dimensional Hopf algebra over $\kk$  up to isomorphism?
\end{introques}

 So far, there is no example of $p^2$-dimensional Hopf algebra $H$ over $\kk$ such that both $H$ and $H^*$ are not pointed. An affirmative answer to the following question will lead to a complete classification of $p^2$-dimensional Hopf algebras over $\kk$, and hence an affirmative answer to Question \ref{Q2}.
\begin{introques}\label{Q3}
Let $\kk$ be an algebraically closed field of characteristic $p > 0$. Must every $p^2$-dimensional Hopf algebra  over $\kk$ be pointed? 
\end{introques}

In this paper, most of our results are stated over an algebraically closed base field $\kk$ of positive characteristic $p$. In the first part of our paper, we study the general structures of the $p^n$-dimensional Hopf algebras over $\kk$ with $n \ge 2$. We first show that any noncosemisimple $p^n$-dimensional Hopf algebra over $\kk$ with $p^{n-1}$ group-like elements must be pointed. Along this line of thoughts, one would like to understand those $p^n$-dimensional Hopf algebras $H$ which admit a $p^{n-1}$-dimensional Hopf subalgebra $\kk[x]$ generated by a primitive element $x$. In this case, assuming $x$ is semisimple and $H$ is nonsemisimple, we show that $H$ has exactly $p^{n-1}$ indecomposable projectives whose dimensions are equal to $p$. Moreover, $H$ is involutive if it is unimodular. When $H$ is $p^2$-dimensional, we are able to show if it contains a primitive element, then its coradical $H_0$ cannot have dimension greater than $p$. As an application, we answer Questions \ref{Q2} and \ref{Q3} affirmatively for $p\le 5$, which completes the classification of all $p^2$-dimensional Hopf algebras over $\kk$ for $p\le 5$.  

For the second part of our paper, we introduce the notion of twist equivalence for cleft extensions to classify the $p^2$-dimensional Hopf algebras, which are extensions of $p$-dimensional ones. We show the second Sweedler cohomology $H^2_{sw}(\kk[G]^*, A)$ is trivial for any elementary abelian $p$-group $G$ which acts as automorphisms on a finite-dimensional commutative algebra $A$ over $\kk$. Consequently, we  describe the corresponding cleft extensions in terms of twist equivalence. Independent of the classification result of pointed Hopf algebras of dimension $p^2$ \cite{WW14}, we provide a list of characterizations of
the Radford algebra $R(p)$. In particular, $R(p)$ is the unique nontrivial Hopf algebra extension of
$\kk[C_p]^*$ by $\kk[C_p]$, where $C_p$ denotes the (multiplicative) cyclic group of order $p$. Finally, we use the classification result \cite{WW14} to identify all the $p^2$-dimensional Hopf algebras with a normal Hopf subalgebra of dimension $p$. 

\subsection*{Acknowledgments} 
 Ng was partially supported by NSF grant
 DMS1664418 and  Wang was partially supported by Simons collaboration grant 
\#688403 and Air Force Office of Scientific Research grant 
FA9550-22-1-0272. This research began in the summer of 2022 during Wang's visit to the first author at Louisiana State University.

\section{Preliminaries and background}
Let $\kk$ be a ground field of arbitrary characteristic. We use $H$ to denote a finite-dimensional Hopf algebra over $\kk$ with counit $\epsilon: H\to \kk$, comultiplication $\Delta: H\to H\otimes H$ and antipode $S: H\to H$. The dual vector space  $H^*$ of $H$ is also a Hopf algebra with the structure induced from that of $H$. The \emph{augmentation ideal} of $H$ is defined to be $H^+:={\rm ker}\, \epsilon$. We use Sweedler's notation to write $\Delta(h)=\sum_{(h)} h_1\otimes h_2$ for all $h\in H$. When the context is clear, we may suppress $\sum_{(h)}$. The readers are referred to \cite{Mont93} and \cite{Sw69} for  more details on Hopf algebras and some fundamental results.

By an $H$-module, we always mean a left $H$-module of finite dimension over $\kk$. For any two $H$-modules $V$ and $W$, their tensor product $V\otimes W$ over the base field $\kk$ admits a natural action of $H$ via $\Delta$. The {\it left dual} $V^*$ of an $H$-module $V$ is the $\kk$-linear space  $V^*=\Hom_\kk(V,\kk)$ equipped with the $H$-module structure given by 
\[
(hf)(v)=f(S(h)v),\quad \text{for all $h\in H$, $f\in V^*$, $v\in V$}.
\]
Given an algebra automorphism $\sigma$ of $H$, one can twist the $H$-action of any $H$-module $V$ to obtain another $H$-module $\!_\sigma V$. More precisely, $\!_\sigma V=V$ as vector spaces with the $\sigma$-twisted action given by
\[
h\cdot_\sigma v=\sigma(h)v,\quad \text{for all $h\in H$, $v\in V$.}
\]
It is immediate to see that $\!_\sigma(\cdot)$ defines a $\kk$-linear equivalence on the category $H$-mod of left $H$-modules. Moreover, the natural isomorphism between $V$ and $V^{**}$ of finite-dimensional vector spaces defines an isomorphism  $\!_{S^2}V\cong V^{**}$ of $\kk$-linear monoidal functors.

We denote by $\irr(H)$ a complete set of nonisomorphic simple $H$-modules. For any $V\in \irr(H)$, the {\it projective cover} of $V$ is denoted by $P(V)$. Since $H$ is a Frobenius algebra \cite{LG}, we have (see \cite[61.13]{CR62})
\begin{equation}\label{dimH}
    \dim H=\sum_{V\in \irr(H)} \dim V\cdot \dim P(V).
\end{equation}
For any $H$-module $M$ and $V\in \text{Irr}(H)$, we define $[M:V]$ to be the multiplicity of $V$ appearing as a composition factor of $M$. Let $V,W\in \text{Irr}(H)$. Since $V^*\otimes P(W)$ is projective, it is a direct sum of indecomposable projective $H$-modules. More precisely,

\begin{align}\label{PM}
V^*\otimes P(W)= & \bigoplus_{U\in \mathrm{Irr}(H)} \dim \left(\Hom_H(V^*\otimes P(W),U)\right)\cdot P(U) \nonumber\\
=&\bigoplus_{U\in \mathrm{Irr}(H)} [V\otimes U:W]\cdot P(U).
\end{align}

Recall that a {\it group-like element} in a Hopf algebra $H$ is a nonzero element $g$ such that $\Delta(g)=g\otimes g$.  The group-like elements in $H$ form a subgroup of the group of units in $H$,  denoted by $G(H)$. Note that the set of 1-dimensional characters of $H$ coincides with the set of group-like elements of the dual Hopf algebra $H^*$, which means $G(H^*)=\Hom_{\rm Alg}(H,\kk)$. For any $\beta\in G(H^*)$, we denote by $\kk_\beta$ the corresponding $H$-module with the underlying $\kk$-linear space equal to $\kk$. In particular, the trivial $H$-module determined by the counit map $\epsilon: H\to \kk$ will simply be denoted by $\kk$.  Moreover, we define an injective group homomorphism $R: G(H^*)\to {\rm Aut}_{\rm Alg}(H)$ such that 
\begin{equation}\label{Rwinding}
R(\beta)(h):=h\leftharpoonup \beta=\sum_{(h)} \beta(h_1)\,h_2,\text{    for all $\beta\in G(H^*)$ and $h\in H$.}
\end{equation}
For any $H$-module $V$, one can check directly that 
\begin{equation}\label{eq:Twisted}
\kk_\beta \otimes V\cong \!_{R(\beta)} V\quad \text{as $H$-modules}.
\end{equation}
An element $x\in H$ is called {\it skew primitive} if $\Delta(x)=x\otimes g+h\otimes x$ for some group-like elements $g$ and $h$ in $H$. In this case, $x$ is also called a {\it $(g,h)$-primitive element}. An $(1,1)$-primitive element is simply called a {\it primitive element}. For any group-like elements $g,h$ of $H$, we use $P_{g,h}(H)$ to denote the subspace of all $(g,h)$-primitive elements in $H$. In particular, the subspace $P(H)=P_{1,1}(H)$ of primitive elements of $H$ is closed under the commutator bracket, and hence a Lie subalgebra of $H$. Moreover, if the base field $\kk$ has characteristic $p$, then $P(H)$ is a {\it $p$-restricted Lie algebra}, where the restricted map $[p]: P(H)\to P(H)$ is given by the $p$th power map in $H$ (see \cite{Jac}). 

The sum of simple subcoalgebras of a coalgebra $C$ is called the \emph{coradical} of $C$ and denoted by $C_0$. A Hopf algebra is called \emph{pointed} if $H_0 = \kk[G(H)]$, and called \emph{connected} if $H_0 = \kk 1_H$.  A finite-dimensional Hopf algebra $H$ is often called \emph{basic} if $H^*$ is pointed. In particular, $H$ is  \emph{local}, which means $H^+$ is the Jacobson radical of $H$, if and only if $H^*$ is connected. The endomorphism ring $\End_H(V)$ of any indecomposable $H$-module $V$ is a local ring (see \cite{AF92, P82}). 
The following trivial lemma is a reminder of these terminologies. 

\begin{lem}\cite[Lem. 2.2]{XT14}\label{lem:local}
The following statements concerning a finite-dimensional Hopf algebra $H$ with counit $\epsilon$ are equivalent. 
\begin{enumerate}[label=\rm{(\roman*)}]
\item $H$ is local;
\item ${\rm ker}\,\epsilon$ is a nilpotent ideal; 
\item  ${\rm ker}\,\epsilon$ is the Jacobson radical of $H$;
\item $H^*$ is connected.
\end{enumerate}
Moreover, $H$ is basic if, and only if, $H^*$ is pointed. 
\end{lem}

Finally, we recall some basic notations on extensions of finite-dimensional Hopf algebras. In the sense of \cite[Definition 1.3]{Masuoka94}, a short exact sequence of finite-dimensional Hopf algebras $K$, $A$ and $H$, denoted by
\begin{equation}\label{E:SES}
1\to K\xrightarrow{\iota} A\xrightarrow{\pi} H\to 1,
\end{equation}
consists of  Hopf algebra homomorphisms $\iota$ and $\pi$ such that $\iota$ is injective, $\pi$ is surjective,  $\iota(K)$ is a normal Hopf subalgebra of $A$ and ${\rm ker}\,\pi=\iota(K^+)A$ (see \cite[Lem.-Definition 1.1]{Masuoka94}). In this case, we say that $A$ is an {\it extension} of $H$ by $K$.  

\begin{lem}\label{lem:SESlocal}
For a short exact sequence of finite-dimensional Hopf algebras described in \eqref{E:SES}, the following hold true. 
\begin{enumerate}[label=\rm{(\roman*)}]
    \item Suppose $K$ is local. Then $\irr(H) = \irr(A)$, and so $H$ is basic if and only if $A$ is basic.
    \item Suppose $H$ is connected. Then, $\irr(K^*) = \irr(A^*)$, and so $K$ is pointed if and only if $A$ is pointed.
    \end{enumerate}
    In particular, we have
    \begin{enumerate}
    \item[\rm (iii)] $A$ is local if and only if both $K$ and $H$ are local.
    \item[\rm (iv)] $A$ is connected if and only if both $K$ and $H$ are connected.
\end{enumerate}
\end{lem}
\begin{proof}
 Without loss of generality, we may assume  $\iota$ is just the inclusion map and so $K$ is 
 a normal Hopf subalgebra of $A$. 
 
 (i) Since $K$ is local, $K^+$ is nilpotent. The normality of $K$ in $A$ implies that $I:={\rm ker}\, \pi=K^+A=AK^+$ is also a nilpotent ideal of $A$. Therefore, simple $A$-modules are simple modules over $A/I\cong H$. Hence the result follows by definition. 

(iii) If $H$ and $K$ are local, then $A$ is local by the proof of (i). Now suppose $A$ is local. Then $A^+$ is nilpotent, and so is $K^+$. Since $H^+=\pi(A^+)$,  $H^+$ is also nilpotent.  So $K$ and $H$ are local by Lemma \ref{lem:local}. 

(ii) and (iv) follow directly from (i) and (iii) by duality.
\end{proof}

\section{Hopf algebras of dimension $p^n$ with $p^{n-1}$ group-like elements}
In this section, we discuss the pointedness of finite-dimensional Hopf algebras with large group-like elements. Though some of our results may apply to arbitrary characteristics, we mainly focus on the positive characteristic case. 

For any module $V$ over a finite-dimensional Hopf algebra $H$, we define 
$$
\stab(V)~=~\{\b \in G(H^*)\mid \kk_\b \o V \cong V \text{ as $H$-modules}\}.
$$ 
Obviously, $\stab(V)$ is a subgroup of $G(H^*)$.
The following result was established over a base field $\kk$ of characteristic zero in \cite[Prop. 2.5]{EG04} and \cite[Lem. 1.4]{Ng08}. We extend its proof to positive characteristics by applying group cohomology for subgroups of $\stab(V)$ for any simple $H$-module $V$. We first observe the following cohomology vanishing lemma. 
\begin{lem} \label{l:vanishing_coh}
    Let $\kk$ be an algebraically closed field of characteristic $p > 0 $, and  $G$ a $p$-group. For any finite-dimensional commutative semisimple algebra $A$ over $\kk$ on which $G$ acts as algebra automorphisms, $H^2(G, A^\times)$ is trivial.  
\end{lem}
\begin{proof}
  For any $\sigma \in Z^2(G, A^\times)$, $\sigma^{|G|}$ is a coboundary by \cite[Thm. 6.5.8]{Wei94}, and so $\sigma^{|G|}=\partial f$ for some function $f: G \to A^\times$. Since $A$ is commutative semisimple, $A=\oplus_{i=1}^n \kk\, e_i$ where $\{e_1,\ldots,e_{n}\}$ is the complete set of primitive orthogonal idempotents for $A$. Note that for any $a \in A$, the polynomial equation $X^{|G|}-a = 0$ has a unique solution in $A$ since $\kk$ has characteristic $p$ and $G$ is a $p$-group. Therefore, there exists $\ol f : G \to A^\times$ such that $\ol f^{|G|}=f$. Thus, $\left(\sigma/\partial \ol f\right)^{|G|} = 1$ or $\sigma = \partial \ol f$. Hence, $H^2(G,A^\times)$ is trivial.  
\end{proof}
\begin{lem}\label{l:freeness}
Let $H$ be a finite-dimensional Hopf algebra over an algebraically closed field $\kk$, and $V$ a simple $H$-module. Suppose there is a subgroup $G \le \stab(V)$ such that $H^2(G, \kk^\times)$ is trivial. Then $V$ admits a Hopf module structure in ${_{\kk[G]}}\!\MM^{H^*}$, and hence is a free $\kk[G]$-module. In particular, if {\rm (i)} ${\rm char}(\kk) = 0$ and $G$ is a cyclic group , or {\rm (ii)} ${\rm char}(\kk) =p$ and $G$ is a $p$-group, then $|G|\mid \dim(V)$.
\end{lem}
\begin{proof}
In view of \eqref{eq:Twisted},  we have $H$-module isomorphisms 
$$
{_{R(\b)}}\!V\cong \kk_\b \o V \cong V \quad \text{for } \b \in G\,.
$$
Let $F_\b: {_{R(\b)}}\!V \to V$ be such an isomorphism of $H$-modules for each $\b \in G$, which are not canonical. Since $R(\e) = \id_H$, we choose $F_\e = \id_V$.  Note that $F_\b \in \GL(V)$ and ${_{R(\b)}}\! ({_{R(\b')}}\!V) = {_{R(\b\b')}}\!V $ (or  $\kk_\b \o \kk_{\b'} \cong  \kk_{\b\b'}$) for any $\b, \b' \in G$. We find $F_{\b}\circ F_{\b'}$ and $F_{\b\b'} :{_{R(\b\b')}}\!V \to V$ are $H$-module isomorphisms. By Schur's lemma, 
$$
F_{\b}\circ F_{\b'} = \g(\b, \b') F_{\b\b'}
$$
for some  $\g(\b, \b') \in \kk^\times$.
The associativity of composition 
\[
(F_{\b}\circ F_{\b'}) \circ F_{\b''}=F_{\b}\circ (F_{\b'}\circ F_{\b''}),
\]
implies $\g(\b',\b'')\g(\b,\b'\b'')=\g(\b\b',\b'')\g(\b,\b')$ for any $\b,\b',\b''\in G$. Since $F_\e = \id_V$,  $\g(\b,\e)=\g(\e, \b)=1$.  Therefore, $\g$ defines a normalized 2-cocycle in $Z^2(G, \kk^\times)$. Since $H^2(G, \kk^\times)$ is trivial,  $\g$ is a coboundary, namely, there exists a map $f: G \to \kk^\times$ such that $\g(\b, \b') = f(\b)f(\b')/f(\b\b')$ with $f(\e)=1$ for all $\b, \b' \in G$. Now we define the $\kk$-linear map $\eta: \kk[G]\to\End_\kk(V)$ by
\begin{align*}
\eta(\b)=F_\b/f(\b)\quad \text{for any}\ \b\in G.
\end{align*}
By the definitions of $\gamma$ and $f$, it is straightforward to check that $\eta$ is indeed an $\kk$-algebra homomorphism. Therefore,  $V$ is a left $\kk[G]$-module via $\eta$, i.e., $\b v = \eta(\b)(v)$ for any $v\in V$ and $\beta\in G$.

Next, the left $H$-action on $V$ defines a right $H^*$-coaction $\rho$ on $V$, namely $h v = \sum_{(v)} v_1(h)v_0$ for $h \in H$ and $v \in V$, where $\rho(v) =  \sum_{(v)} v_0\otimes v_1\in V\otimes H^*$. This makes $V$ simultaneously a left $\kk[G]$-module and a right $H^*$-comodule. We claim that $V$ is a Hopf module in ${_{\kk[G]}}\!\MM^{H^*}$, and hence by the Nichols-Zoeller Theorem (cf. \cite{NZ}), $V$ is a free $\kk[G]$-module. It suffices to show that
$$
\sum_{(\b v)}(\b v)_0 \o (\b v)_1 = \sum_{(v)}\b v_0 \o \b v_1 \quad \text{or}\quad \sum_{(\b v)}  (\b v)_1(h) (\b v)_0 = \sum_{(v)} (\b v_1)(h) \b v_0
$$
for all $\b \in G$, $v \in V$ and $h \in H$. Note that $\eta(\b)$ is an $H$-module isomorphism from $\!_{R(\b)}V$ to $V$. So our claim follows from this calculation: 
\begin{align*}
\sum_{(\b v)} (\b v)_1(h)(\b v)_0&=h(\b v)=\b\left(R(\b)(h)v\right)=\b\left(\sum_{(v)} v_1(R(\b)(h))v_0\right)\\
&=\b\left(\sum_{(v)} v_1\left(\sum_{(h)}\beta(h_1)h_2\right)v_0\right)=\sum_{(v)}\left(\sum_{(h)}\b(h_1)v_1(h_2)\right)\b v_0\\
&=\sum_{(v)} (\b v_1)(h)(\b v_0).
    \end{align*}

If $G$ is a cyclic group and ${\rm char}(\kk)=0$, $H^2(G, \kk^\times)$ is well-known to be trivial (cf.  \cite[Thm. 6.2.2]{Wei94}). If ${\rm char}(\kk)= p$ and $G$ is a $p$-group, then $H^2(G, \kk^\times)$ is trivial by Lemma \ref{l:vanishing_coh}. Now, the final assertion follows directly from the first statement.
\end{proof}

Our next lemma is derived from the proof of \cite[Lem. 1.3]{NW19}.

\begin{lem} \label{l:dim}
Let $\kk$ be an algebraically closed field of characteristic $p>0$, $H$ a nonsemisimple Hopf algebra of dimension $p^n$ over $\kk$ for some integer $n \ge 1$, and $P$ an indecomposable projective $H$-module. Suppose  $P \cong \kk_\b\o P^{**}$ for some $\b \in G(H^*)$. Then, 
\begin{enumerate}[label=\rm{(\roman*)}]
    \item $\dim P$ is a positive even integer if $p=2$;
    \item $\dim P \ge p$ if $\dim P$ and $p$ are odd.
\end{enumerate}
\end{lem}
\begin{proof}  
Note that $\kk_\b\o P^{**}\cong \!_\sigma P$, where $\sigma=R(\b) \circ S^2 = S^2 \circ R(\b)$. So there exists an $H$-module isomorphism $\phi: P \to \!_\sigma P$. By the celebrated formula of Radford \cite{Rad76}, $S^{4p^n}=\id_H$. Hence $\sigma^{2p^n}=R(\b)^{2p^n}\circ S^{4p^n}=R(\b^{2p^n})\circ S^{4p^n}=\id_H$. Thus, $\phi^{2p^n}$ is an $H$-module automorphism of $P$. Since $P$ is indecomposable, $\End_H(P)$ is a local $\kk$-algebra and so $\phi^{2 p^n} - c \id_P$ is nilpotent for some $c \in \kk^\times$. By rescaling, we can assume that both $\phi^{2 p^n}$ and $\phi^2$ are unipotent. Thus, the eigenvalues of $\phi$ are $\pm 1$ if ${\rm char}(\kk)=p> 2$, and are $1$ if $p=2$. 

We claim $\tr(\phi)=0$. If $\b$ is trivial, it follows from \cite[Lem. 1.1]{NW19}. If $\b$ is nontrivial, it follows from \cite[Lem. 1.2]{Ng08}. Thus, if $p=2$, $\tr(\phi)=0$ implies $\dim P$ is a positive multiple of 2. 

Now, we assume $p > 2$ and $\dim P$ is odd. Let $d_\pm$ be the multiplicities of the eigenvalues $\pm 1$ for $\phi$, respectively. Since 
$$
0 = \tr(\phi)=d_+ - d_-,
$$
$|d_+ - d_-| = mp$ for some nonnegative integer $m$. If $m =0$, then  $\dim(P) = 2d_+$, which is even. Since $\dim(P)$ is odd, $m \ge 1$ and hence $\dim(P) = d_+ + d_- \ge p$.
\end{proof}

\begin{thm}\label{thm:basic}
Let $\kk$ be an algebraically closed field of characteristic $p > 0$ and $H$ a nonsemisimple Hopf algebra of dimension $p^n$ over $\kk$ for some integer $n \ge 1$. If $|G(H^*)|=p^{n-1}$, then $H^*$ is pointed. 
\end{thm}
\begin{proof}
We first consider $p=2$. Suppose $H^*$ is not pointed. Then the sum of all the simple subcoalgebras of $H^*$ with dimensions $> 1$ is a nontrivial subcoalgebra $C$ of $H^*$. Note that $C$ is a free $\kk[G(H^*)]$-module since it is a ${}_{\kk[G(H^*)]}\MM^{H^*}$ Hopf module.  Thus, 
$$
\dim (H_0^*) =  |G(H^*)| +\dim C \ge  2|G(H^*)|=2^n = \dim H^*\,.
$$
Hence $H^*$ is cosemisimple, or $H$ is semisimple.  This proves the statement for $p=2$.

Now, we assume $p >2$. By Lemma \ref{lem:local}, it suffices to show that $\irr(H)=\{\kk_\beta\,|\, \beta\in G(H^*)\}$. It is easy to see that $P(\kk_\b\o V)\cong \kk_\b \o P(V)$ and $P(V^{**})\cong P(V)^{**}$ for any $V\in \irr(H)$. In particular, $\dim P(\kk)=\dim P(\kk_\b)$ for all $\b \in G(H^*)$. 

If $\dim P(\kk)$ is odd, then  $\dim P(\kk) \ge p$ by Lemma \ref{l:dim}. So \eqref{dimH} implies that 
$$
p^n = \dim H \ge \sum_{\b \in G(H^*)} \dim \kk_\beta\cdot \dim P(\kk_\b) \ge p^{n-1} p \,.
$$
Hence, $\dim P(\kk)=p$ and $\irr(H)=\{\kk_\beta\,|\, \beta\in G(H^*)\}$.  

To complete the proof of the theorem, it suffices to show that $\dim P(\kk)$ must be odd.  Let  $K=\langle S^2\rangle\subset {\rm Aut}_\kk(H)$. Since $S^{4p^n}=\id_H$, we have $|K|$ divides $2p^n$. For $\beta \in G(H^*)$, $R(\beta)$ is an algebra automorphism of $H$ via \eqref{Rwinding}. Thus, both $G(H^*)$ and $K$ act on $\irr(H)$ by twisting the $H$-actions on irreducible modules through their associated algebra automorphisms of $H$.

Since $R(\beta)\circ S^2=S^2\circ R(\beta)$ for any $\beta\in G(H^*)$, the above two actions commute with each other. Thus,   the set of $K$-orbits in $\irr(H)$ admits a $G(H^*)$-action induced by its action of $\irr(H)$.  Let $\mathcal O_1,\ldots,\mathcal O_r$ be the $K$-orbits, and  $V_1,\ldots,V_r$ be representative irreducible $H$-modules in these $K$-orbits, respectively. Then \eqref{dimH} implies that
\begin{equation}\label{them2.5H}
    p^n=\dim H=\sum_{i=1}^r |\mathcal O_i|\,\dim(V_i)\,\dim P(V_i).
\end{equation}
Since $\dim H$ is odd, there exists some index $i_0$ such that $|\mathcal O_{i_0}|$, $\dim V_{i_0}$ and $\dim P(V_{i_0})$ are all odd. Without loss of generality, we may assume $i_0=1$.

If we could show that $\dim V_{1}=1$, then $V_{1} \cong \kk_\b$ for some $\b\in G(H^*)$, and hence $\dim P(\kk)= \dim P(\kk_\b)$ is odd. Therefore, it suffices to show that $\dim V_1 =1$.

 Since $|\mathcal O_1|$ is odd and $|\OO_1|\mid |K|\mid 2p^n$,  $|\mathcal O_1|=p^k$ for some nonnegative integer $k \le n$. Let  
 $$
 G = \{\b \in G(H^*)\mid \kk_\b \o V_1 \in \OO_1 \}
 $$
and $G_1$ be the stabilizer of $V_1$. Obviously, $G_1 \subseteq G$, and   $|G_1| \mid \dim(V_1)$ by Lemma \ref{l:freeness}.
Since the $K$-action and the $G(H^*)$-action on $\irr(H)$ commute,  $\OO_1$ is closed under the action of $G$ and $K$ acts transitively on the set of $G$-obits in $\OO_1$. Consequently, all the $G$-orbits in $\OO_1$ have the same size $\frac{|G|}{|G_1|}$ and we have
\begin{equation}\label{them2.5L}
\ell \cdot \frac{|G|}{|G_1|}=|\mathcal O_1|=p^k,
\end{equation}
where $\ell$ is the number of $G$-orbits in $\mathcal O_1$. In particular, $\ell$ is a power of $p$. 

Now suppose $\dim V_{1}>1$. Let $\{\b \OO_1 \mid \b \in G(H^*)\}=\{\mathcal O_1,\ldots,\mathcal O_t\}$ for some nonnegative integer $t\leq r$ by relabelling the indices of the $K$-orbits. Then \eqref{them2.5H} and \eqref{them2.5L} imply that 
\begin{align*}
   p^n  = \dim(H) &\ge \sum_{\b  \in G(H^*)} \dim(\kk_\beta)\, \dim P(\kk_\b) + \sum_{1\le i\le t} |\mathcal O_i|\,\dim (V_i)\,\dim P(V_i) \\
   & = |G(H^*)| \dim P(\kk)+\frac{|G(H^*)|}{|G|} |\mathcal O_1|\,\dim (V_1)\dim P(V_1)\\
   &=p^{n-1}\left(\dim P(\kk)+\ell\,\dim (P(V_1))\, \frac{\dim (V_1)}{|G_1|}\right).
\end{align*}
Therefore, 
\begin{equation} \label{eq:last_ineq}
    p=\dim P(\kk)+\ell\,\dim (P(V_1))\, \frac{\dim (V_1)}{|G_1|} > \ell \dim P(V_1)
\end{equation}
and hence $\ell < p$. It follows from 
\eqref{them2.5L} that $\ell=1$, or $G$ acts transitively on $\OO_1$. Since both $V_1$ and $V_1^{**}$ are in the $K$-orbit $\mathcal O_1$, there exists $\b\in G$ such that $V_1^{**} \cong \kk_\b \o V_1$. Hence $\kk_{\b^{-1}} \o P(V_1)^{**} \cong P(V_1)$. It  follows from Lemma \ref{l:dim} that $\dim P(V_1) \ge p$, which contradicts \eqref{eq:last_ineq}. Therefore, $\dim V_1 = 1$.
\end{proof}

The dual statement of the preceding theorem is stated as follows.
\begin{thm}\label{thm:lageG}
Let $\kk$ be  an algebraically closed field of characteristic $p >0$ and $H$ a non-cosemisimple Hopf algebra of dimension $p^n$ over $\kk$ for some integer $n \ge 1$. If $|G(H)|=p^{n-1}$, then $H$ is pointed. 
\end{thm}

\section{Finite-dimensional Hopf algebras with primitive elements}
In this section, we turn our attention to the existence of primitive elements in finite-dimensional Hopf algebras over an algebraically closed field $\kk$. It is well-known that finite-dimensional Hopf algebras in characteristic zero do not possess nonzero primitive elements \cite[Prop. 1(b)]{Stefan97}. One may assume, a priori, all results about primitive elements are asserted over positive characteristics.   

We first observe the following lemma regarding extensions of simple $H$-modules, and we simply denote $H^+$ by $\fm$.  

\begin{lem}\label{lem:ext}
Let $H$ be a finite-dimensional Hopf algebra over an algebraically closed field $\kk$, and $L$ any simple $H$-module. Denote by $\mathfrak n=Ann(L)$ the annihilator ideal of $L$ in $H$. Then we have
\[
\dim L \cdot \dim \Ext_H^1(L,\kk)~=~\dim\left(\mathfrak m\cap \mathfrak n/\mathfrak m\mathfrak n\right).
\]
\end{lem}
\begin{proof}
Applying $\Ext_H^1(-,\kk)$ to the short exact sequence $0\to \mathfrak n\to H\to H/\mathfrak n\to 0$,  we obtain the  exact sequence
 \[
 0\to \Hom_H(H/\mathfrak n,\kk)\to \Hom_H(H,\kk)\to \Hom_H(\mathfrak n,\kk)\to \Ext_H^1(H/\mathfrak n,\kk)\to 0.
 \]
Note that $H/\mathfrak n=H/Ann(L)\cong \bigoplus_{\dim L} L$ as $H$-modules. Also we have 
\[\Hom_H(V, \kk)~\cong~\Hom_\kk(V/\mathfrak{m} V, \kk)~\cong~ (V/\mathfrak{m} V)^*\] 
as vector spaces for any $H$-module $V$. Thus,
\begin{align*}
&\dim L\cdot \dim \Ext_H^1(L,\kk)\,=\,\dim \Ext_H^1(H/\mathfrak n,\kk)\\
=&\, \dim \Hom_H(H/\mathfrak n,\kk)+\dim \Hom_H(\mathfrak n,\kk)-\dim \Hom_H(H,\kk)\\
=&\, \dim (H/(\mathfrak m+\mathfrak n))+\dim (\mathfrak n/\mathfrak m\mathfrak n)-\dim (H/\mathfrak m)\\
=&\, \dim (\mathfrak n/\mathfrak m\mathfrak n)-\dim ((\mathfrak m+\mathfrak n)/\mathfrak m)\\
=&\, \dim (\mathfrak n/\mathfrak m\mathfrak n)-\dim (\mathfrak n/\mathfrak m\cap \mathfrak n)\\
=&\, \dim (\mathfrak m\cap \mathfrak n/\mathfrak m\mathfrak n).\qedhere
\end{align*}
\end{proof}
Note that if $L=\kk$, the condition $\kk$ being algebraically closed in the preceding lemma can be dropped, which applies to the following two generally known lemmas (e.g., see 
\cite{GZ10}).  We give their proofs for the sake of completeness. 

\begin{lem}\label{lem:dext}
Let $H$ be a finite-dimensional Hopf algebra over any field $\kk$. Then 
\[\dim \Ext_H^1(\kk,\kk)~=~\dim  (\mathfrak m/\mathfrak m^2)~=~\dim P(H^*).\]
\end{lem}
\begin{proof}
The first equality in the statement follows from Lemma \ref{lem:ext}. Next, consider the subcoalgebra $C=\pi^*(H/\mathfrak m^2)^*$ of $H^*$, where $\pi: H \to H/\mathfrak{m}^2$ is the natural algebra surjection. For $f \in C$, $\Delta_{H^*}(f)(\mathfrak{m} \o \mathfrak{m}) = f(\mathfrak{m}^2) = 0$.
This implies that $\Delta_{H^*}(f) \in \epsilon \otimes H^*+H^*\otimes \epsilon$. 

Let $I = 1^\perp \cap C$. Then $\dim I = \dim H -\dim \mathfrak{m}^2 - 1 = \dim(\mathfrak m/\mathfrak m^2) $. We will complete the proof by showing $I = P(H^*)$. For $f \in I$, $\Delta_{H^*}(f)= \e \o u + v \o \e + c \e \o \e$ for some $c \in \kk$ and $u, v \in I$. The counit property of $\Delta_{H^*}$ implies $f = v+c \e = u + c \e$. Since $f, u, v \in 1^\perp$, $c=0$ and $f = u =v$. Thus, $f \in P(H^*)$ and hence $I \subseteq P(H^*)$.  Conversely, if $f \in P(H^*)\setminus\{0\}$,  then $f(1)=0$ and $\Delta_{H^*}(f) = \e \o f + f \o \e$. For any $x,y \in \mathfrak{m}$, $f(xy) = \Delta_{H^*}(f)(x \o y) = 0$. Thus, $f \in I$ and hence $P(H^*) \subseteq I$. 
\end{proof}

We recall the associated graded algebra
\begin{equation}\label{eq:gr}
{\rm gr}_\mathfrak m H~:=~\bigoplus_{i=0}^\infty \mathfrak m^i/\mathfrak m^{i+1}
\end{equation}
with respect to the $\fm$-adic filtration in $H$.

\begin{lem}
Let $H$ be a finite-dimensional Hopf algebra over $\kk$ of characteristic $p>0$. Then $\bigcap_{i\ge 0} \mathfrak m^i$ is a Hopf ideal of $H$ with the quotient Hopf algebra $H/\bigcap_{i\ge 0} \mathfrak m^i$ of dimension $p^n$ for some $n\ge \dim P(H^*)$. In particular, if ${\rm gr}_\mathfrak mH$ is commutative, then so is $H/\bigcap_{i\ge 0} \mathfrak m^i$.
\end{lem}
\begin{proof}
The assertion that $\bigcap_{i\ge 0} \mathfrak m^i$ is a Hopf ideal of $H$ and that ${\rm gr}_\mathfrak m H$ is a graded Hopf algebra follows directly from \cite[Lem. 3.3]{GZ10}.  The graded algebra ${\rm gr}_\mathfrak mH$ is connected, cocommutative, and generated by degree one elements, which are primitive. Thus, ${\rm gr}_\mathfrak mH$ is isomorphic to the restricted universal enveloping algebra $ u(\mathfrak g)$ (cf.  \cite[13.2.3]{Sw69},\cite[Thm. 1]{May}), where $\mathfrak g=P({\rm gr}_\mathfrak mH)\supseteq \mathfrak m/\mathfrak m^2$. 
Hence we have \[\dim H\big /\bigcap_{i\ge 0} \mathfrak m^i=\dim {\rm gr}_\mathfrak mH=p^{\dim \mathfrak g}\ge p^{\dim (\mathfrak m/\mathfrak m^2)}=p^{\dim P(H^*)}\]
by Lemma \ref{lem:dext}. The last conclusion is \cite[Lem. 3.5]{GZ10}.
\end{proof}

\begin{prop} \label{p:primitive1}
Let $H$ be a finite-dimensional Hopf algebra over an algebraically closed  field $\kk$ of characteristic $p >0$, which contains a nonzero primitive element $x$. Suppose $C =\sum_{i} C_i$ is a sum of simple subcoalgebras of $H$.
\begin{enumerate}[label=\rm{(\roman*)}]
    \item If $xC \subset C$, then any simple right $C$-comodule $M$ admits a left $\kk[x]$-action such that $M$ is a Hopf module in $_{\kk[x]}\!\MM^H$. In particular, $\dim (\kk[x]) \mid \dim (M)$ and hence $\dim (\kk[x])^2 \mid \dim (C_i)$ for all $i$.
    \item If $C$ is simple and $p \nmid \dim(C)$, then $\kk[x]C$ is a free left $\kk[x]$-module of rank $\dim(C)$. 
\end{enumerate}
\end{prop}
\begin{proof}
(i) Since $xC\subset C$, we can let $x$ act on the dual algebra $C^*$ via usual dual module action, namely $(x g)(c)=g(S(x)c)=-g(xc)$ for any $c\in C,g\in C^*$. Furthermore, one can check that $x$ acts as a $\kk$-derivation on $C^*$ for $x$ is primitive. Since $C^*$ is separable, $HH^1(C^*, C^*)$ is trivial. Therefore, $x$ acts as an inner derivation on $C^*$, which means there exists $f_x \in C^*$ such that $x g = [f_x, g] = \ad(f_x)(g)$ for all $g \in C^*$. Note that $f_x\in C^*$ is uniquely determined by $x$ up to the center of $C^*$. Since $C^* = \bigoplus_{i} C^*_i$ as algebra, $f_x = \sum_i f_i$ where $f_i \in C_i^*$, where $f_i$ is uniquely determined up to a scalar multiple of $1_{C_i^*}$. By definition we obtain that $xc=\sum_{(c)} f_x(c_2) c_1-f_x(c_1) c_2 =\sum_{(c)} f_i(c_2) c_1-f_i(c_1) c_2 $  for $c \in C_i$. In particular, $xC_i \subset C_i$ for each $i$.

Suppose $M$ is a simple right comodule  of $C$. Then $M$ is also a simple right comodule of $C_i$ for some $i$, and hence a $C_i^*$-module given by
$$
g \cdot m ~= ~\sum_{(m)} m_0\, g(m_1)\,
$$
for all $m\in M$ and $g\in C_i^*$. We claim that one can choose a suitable $f_i\in C_i^*$ to define an  $\kk[x]$-module action on $M$ given by $x \cdot m := f_i \cdot m$ for $m \in M$. Since $\kk[x]$ is primitively generated, so $\kk[x]\cong u(\mathfrak g)$ with $\mathfrak g={\rm span}_\kk(x,x^p,x^{p^2},\ldots)$. Thus $x$ satisfies the following relation in $\kk[x]$ 
\begin{equation}\label{lemmaeq:xC}
x^{p^n}+a_{n-1}x^{p^{n-1}}+\cdots+a_1x^{p}+a_0x=0
\end{equation}
for some coefficients $a_0,\ldots,a_{n-1}\in \kk$. Then
\begin{align*}
0 & = \ad(f_x)^{p^n}+a_{n-1}\ad(f_x)^{p^{n-1}}+\cdots+a_1(\ad f_x)^{p}+a_0\ad(f_x)  \\
  & =\, \ad\left(f_x^{p^n}+a_{n-1}f_x^{p^{n-1}}+\cdots+a_1 f_x^{p}+a_0f_x\,\right),
\end{align*}
which is equivalent to 
$$
0  = \ad\left(f_j^{p^n}+a_{n-1}f_j^{p^{n-1}}+\cdots+a_1 f_j^{p}+a_0f_j\,\right)\quad \text{ for all }j. 
$$
This implies that $f_j^{p^n}+a_{n-1} f_j^{p^{n-1}}+\cdots+a_1 f_j^{p}+a_0 f_j=\lambda_j\, 1_{C^*}$ for some  scalar $\lambda_j \in \kk$. For each $j$, replacing $f_j$ by $f_j+\eta_j\,1_{C_j^*}$ with  $\eta_j \in \kk$ satisfying  $\eta_j^{p^n}+a_{n-1}\eta_j^{p^{n-1}}+\cdots+a_1\eta_j^{p}+a_0\eta_j=-\lambda_j$, we have $f_x$ satisfy the same relation as $x$ in \eqref{lemmaeq:xC}. This makes the $\kk[x]$-action:
$$
x\cdot m = \sum_{m} m_0 f_i(m_1) = \sum_{m} m_0 f_x(m_1), \quad \text{ for }m\in M,
$$ 
well-defined on $M$.

With this $\kk[x]$-module action on $M$, one can check directly that $M$ is a Hopf module in $\!_{\kk[x]} \MM^H$. This amounts to show that
\[
\sum_{(u\cdot m)} (u\cdot m)_0\otimes (u\cdot m)_1~=~\sum_{(u),\,(m)} u_1\cdot m_0\otimes u_2m_1
\]
for any $u\in \kk[x]$ and $m\in M$. Since $\kk[x]$ is generated by $x$, it suffices to check the above equation for $u=x$. Indeed, we have
\begin{align*}
\sum_{(x),(m)} x_1\cdot m_0\otimes x_2m_1&=\sum_{(m)} x\cdot m_0\otimes m_1+m_0\otimes xm_1\\
&=\sum_{(m)} m_0\otimes f_x(m_1)m_2+ m_0\otimes (f_x(m_2)m_1-f_x(m_1)m_2)\\
&=\sum_{(m)} m_0\otimes m_1f_x(m_2)\\
&=\sum_{(x\cdot m)} (x\cdot m)_1\otimes (x\cdot m)_2.
\end{align*}
Therefore, $M$ is a free $\kk[x]$-module by the Nichols-Zoeller Theorem. So $\dim (\kk[x]) \mid \dim(M)$, and hence $\dim (\kk[x])^2 \mid \dim(M)^2 =\dim(C_i)$.

(ii) It follows from (i) that  $xC \not\subset C$. Then, the comultiplication $\Delta$ defines  a natural nontrivial $C$-$C$-bicomodule structure on $(xC+C)/C$. Since every simple $C$-$C$-bicomodule has dimension $\dim(C)$ and $\dim(xC+C)/C\le \dim(C)$, we find $\dim(xC+C)/C= \dim(C)$. Hence $xC + C = xC \oplus C$. By repeating the argument, one can inductively show that 
\[\kk[x]\,C ~=~ \bigoplus\limits_{i=0}^n x^i\, C\quad \text{and}\quad \dim (x^i\,C)~ =~\dim(C)\]
where $x^iC\not\subseteq C\oplus xC\oplus\cdots \oplus x^{i-1} C$ for all $i\le n$. Thus, $\dim(C) \mid \dim(\kk[x] C)$. It is clear that $\kk[x]C$ is a Hopf module in $\!_{\kk[x]} \MM^H$, and so $\kk[x]C$ is a free  $\kk[x]$-module. In particular
$\dim(\kk[x]) \mid \dim(\kk[x] C)$. Therefore, $\dim(\kk[x])\dim(C) \mid \dim(\kk[x] C) \le \dim(\kk[x])\dim(C)$. Hence, $\dim(\kk[x])\dim(C) =\dim(\kk[x] C)$. Therefore, the $\kk[x]C$ is a free $\kk[x]$-module of rank $\dim(C)$.  
\end{proof}
\begin{cor}\label{cor:cad}
If $H$ is a $p^2$-dimensional Hopf algebra over an algebraically closed  field $\kk$ of characteristic $p>0$, which admits a nonzero primitive element $x$, then $\dim(H_0) \le  p$.
\end{cor}
\begin{proof}
Let $K = \kk[x]$. If $K = H$, then $H_0 = \kk$, and the statement holds vacuously. Suppose $K \ne H$. Then $\dim(K) = p$ by Nichols-Zoeller Theorem since $K$ is a Hopf subalgebra of $H$. Let $C$ be any simple subcoalgebra of $H$. It is clear that $p\nmid \dim C$ for $\dim H=p^2$. By  Proposition \ref{p:primitive1} (ii), $\dim(KC) = p\dim(C)$.  Note that \[
C\subseteq C+xC\subseteq C+xC+x^2C\subseteq \cdots\subseteq KC
\]
is a coalgebra filtration on $KC$. So \cite[Lem. 5.3.4]{Mont93} implies that $KC$ is an irreducible subcoalgebra of $H$ with $(KC)_0 =C$ (see \cite[Definition 5.6.1]{Mont93}). Thus, 
\[\sum_{C} KC~ = ~\bigoplus_{C} KC,\] where the sum is over the set of all simple subcoalgebra $C$ of $H$. 
Hence 
$$
p^2=\dim(H) \ge \dim\left(\bigoplus_{C} KC\right) = p\left(\sum_C \dim C\right)=p\dim H_0,
$$
which yields that $p \ge \dim(H_0)$.
\end{proof}

\begin{prop}\label{prop:semiprim}
Let $H$ be a nonsemisimple $p^n$-dimensional Hopf algebra over an algebraically closed  field $\kk$ of characteristic $p >2$. If $H$ admits a semisimple primitive element $x$ such that $\dim \kk[x] = p^{n-1}$,  then $S^{2p} = \id$ and $He$ is a $p$-dimensional indecomposable projective $H$-module for any primitive idempotent $e \in \kk[x]$. In addition, if $H$ is unimodular, then $S^2=\id$.
\end{prop}
\begin{proof}
   Write $K=\kk[x]$, a connected Hopf subalgebra of $H$ generated by the primitive element $x$.  By \cite[Thm. 1]{May}, $K\cong u(P(K))$. Because $x$ is semisimple, by Hochschild's theorem \cite{Ho54}, $K \cong \kk[G]^*$ as Hopf algebras, where $G \cong C_p^{n-1}$. In particular, $K$ admits a complete set of orthogonal idempotents $e_1, \dots, e_{p^{n-1}}$ and $S^2(e_i)=e_i$ for all $i$. Since $H$ is a free $K$-module of rank $p$, $\dim(He_i) = p$ for all $i$.  Since $H$ is not semisimple, $\sum_{(\Lambda)} S(\Lambda_2)S^2(\Lambda_1)=\e(\Lambda)1=0$, where $\Lambda$ is a nontrivial left integral of $H$. By  \cite[Lem. 3(b)]{RS}, $\tr(r(e_i) \circ S^2) = 0$ for all $i$, where $r(e_i)$ denotes the operator on $H$ given by the right multiplication of $e_i$. Hence $\tr(r(e_i) \circ S^2)=\tr(S^2|_{He_i})=0$. By Radford's $S^4$-formula, we know $S^{4p^n}=\id$ and $(S^4-\id)^{p^n}=0$. Then we can write $S^2|_{He_i} = \pm  \id_{He_i}+N_i$ for some nilpotent operator $N_i$ of $He_i$. However, $e_i \in He_i$ and $S^2(e_i)=e_i$. Therefore, $S^2|_{He_i} = \id_{He_i}+N_i$ for all $i$. Since $\dim He_i = p$, $N_i^p=0$ and so $S^{2p}|_{He_i} = \id_{He_i}$. Therefore, $S^{2p} = \id$. Note that, for each $i$, $(He_i)^{**}\cong \!_{S^2}(He_i)\cong HS^2(e_i)=He_i$. 
   
    Now we show the $He_i$ is indecomposable. Assume to the contrary. Then for any indecomposable summand $P$ of $H e_i$,  $ 1< \dim P < p$ since  $H$ is nonsemisimple. Since $(He_i)^{**} \cong He_i$, $D^j(P)$ is a summand of $He_i$ for $j=0, \dots, p-1$, where $D(V)=V^{**}$. This implies $P^{**} \cong P$ for all indecomposable summands of $He_i$. Since $\dim He_i$ is odd, there must be an indecomposable summand of $He_i$, which has an odd dimension, but this contradicts Lemma \ref{l:dim}(ii).  
   
    For the last assertion, it suffices to show that $S^{4}={\rm id}$. If $H$ contains a nontrivial group-like element $g$, then $S^2={\rm id}$ since it holds for $g,x$ which generate $H$. On the other hand, if $G(H)=1$, then $H^*$ is unimodular. It follows immediately from Radford's $S^4$-formula such that $S^4=\id$ since both $H$ and $H^*$ are unimodular.
\end{proof}

\section{Affirmative answer to Questions \ref{Q2} and \ref{Q3} for $p\le 5$}
As an application of our previous results, in this section, we give an affirmative answer to our main Questions \ref{Q2} and \ref{Q3} for primes $p\le 5$. 

The ensuing result is a natural extension of \cite[Lem. 2.3]{NW19} to dimension $p^2$. Note that  the group algebra $\kk[G]$ for any  $p$-group $G$ of order $p^n$ is a local $p^n$-dimensional Hopf algebra over $\kk$ of characteristic $p$.

\begin{lem}\label{noIRR}
Let $\kk$ be an algebraically closed field of characteristic $p>2$. For any non-local Hopf algebra of dimension $p^2$, we have $|{\rm Irr}(H)|> 2$. Moreover, if $p=5$, then $|{\rm Irr}(H)|> 3$. 
\end{lem}
\begin{proof}
Set ${\rm Irr}(H)=\{\kk=V_0,\ldots, V_n\}$ with $|{\rm Irr}(H)|=1+n \ge 2$. For simplicity, we write $(\dim V_i,\dim P(V_i))=(d_i,D_i)$, $N_{ij}^k=[V_i\otimes V_j: V_k]$ for all $0\le i,j,k\le n$.  Note that $[V_i \o V_j: V_k] = \dim \Hom_H(P(V_k), V_i \o V_j)$ and that  \eqref{dimH} and \eqref{PM} 
can be written as 
\begin{align}
   p^2&~=~\sum_{j=0}^n d_jD_j\,,\label{E1} \\
d_iD_k&~=~\sum_{j=0}^n N_{ij}^kD_j\,.\label{E2} 
\end{align}

Suppose $n=1$. Since $p>2$, we have $d_1>1$ otherwise $G(H^*)=2\nmid p^2$. Thus \eqref{E1} and \eqref{E2} implies that  
\begin{align*}
p^2=D_0+d_1D_1\ \text{and}\ d_1D_1=D_0+N_{11}^1D_1.
\end{align*} 
We get $D_1\mid D_0\mid p^2$. Hence $D_1=p$ and $D_0=kp$ for some $k\ge 1$. Since $d_1=p-k<D_1$, we must have $D_1\ge 2d_1$, which implies that $k>p/2$. So $N_{11}^1=(d_1D_1-D_0)/D_1=(p-2k)<0$, a contradiction!

Next suppose $n=2$ and $p=5$. We first treat the case  $V_1^*\cong V_2$ (hence $V_2^*\cong V_1$). This implies that $d_1=d_2>1$ and $D_1=D_2$. Moreover, \eqref{E1} and \eqref{E2} implies that  \begin{align*}
p^2=D_0+2d_1D_1, \ d_1D_1= d_2 D_1 = D_0+(N_{11}^1+N_{12}^1)D_1.
\end{align*} 
Similar calculation yields that $D_1=p$, $D_0=kp$ for $k\ge 1$ and $d_1=(p-k)/2\ge 2$. Since $p=5$, we must have $k=1$ and $d_1=2$, $D_0=D_1=5$. However,
$$
4 =\dim (V_2 \o V_1) = N_{21}^0 + 2\cdot (N_{21}^1+ N_{21}^2)  = N_{21}^0  + 4\cdot N^1_{21} \ge 1+4\cdot N^1_{21}
$$
since $V_2 \o V_1$ is self-dual and $N_{21}^0 \ge 1$. This forces $N^1_{21} = 0$ and $N_{21}^0 = 4$ but this leads to the absurd inequality 
$$
2p = d_2 D_0 \ge N_{21}^0 D_1 = 4p\,. 
$$

It remains to treat the case such that $V_i^*\cong V_i$ for $i=1,2$. Note that we still have $d_1,d_2>1$. For all $0\le i,j,k\le 2$, we have 
\begin{align*}
N_{ij}^k=[V_i\otimes V_j:V_k]=[(V_i\otimes V_j)^*: V_k^*]=[V_j^*\otimes V_i^*: V_k^*]=[V_j\otimes V_i:V_k]=N_{ji}^k. 
\end{align*}
Subtracting  \eqref{E2} for $i=k=2$ from \eqref{E1} and  letting $\{i,k\} = \{1,2\}$ in \eqref{E2}, we get 
\begin{align}
 p^2&~=~(2d_2-N_{22}^2)D_2+(d_1-N_{21}^2)D_1,\label{E5}\\
 N_{11}^2D_1&~=~(d_1-N_{12}^2)D_2,\label{E3}\\
N_{22}^1D_2 &~=~(d_2-N_{21}^1)D_1.\label{E4}
\end{align}
 If $d_1=N_{21}^2$, then \eqref{E5} implies that $D_2=p$ and $d_2=(p+N_{22}^2)/2\ge \frac{p+1}{2} > D_2/2$. Therefore, $d_2 = D_2$ or $V_2$ is projective, but this is absurd as $d_2D_2=p^2$.
 By the same argument,   $d_2=N_{21}^1$ is also impossible. Therefore,  $d_1>N_{21}^2$ and $d_2>N_{21}^1$. By \eqref{E5}, $\gcd(D_1,D_2)\mid p^2$. If $\gcd(D_1,D_2)=1$, \eqref{E3} and \eqref{E4} imply that $D_1|d_1-N_{12}^2$ and $D_2|d_2-N_{21}^1$. Since $D_i \ge  d_i > 0$, $N_{12}^2 = N_{21}^1=N_{12}^1 =0$. Therefore, all the composition factor of $V_1 \o V_2$ are isomorphic to $V_0$, and hence
$$
0 \ne  \Hom_H(V_1 \o V_2: \kk) \cong  \Hom_H(V_1, V_2)
$$
which is absurd. Thus, we must have $\gcd(D_1,D_2)=p$. By \eqref{E1} again, we conclude that $D_0=D_1=D_2=5$ and $d_1=d_2=2$.  This implies $[P(V_i):V_0]=1$ for $i=1,2$. By duality, $[P(V_0): V_i]=1$ for $i=1,2$, which leads  
\[5=\dim P(V_0)=[P(V_0):V_0]+2([P(V_0):V_1]+[P(V_0):V_2])\ge 2+2\cdot 2=6,
\]
a contradiction again! 
\end{proof}

In \cite{St,Na}, the quantum coordinate ring $\mathcal O_q(SL_2)$ was used to show the nonexistence of $S$-invariant 4-dimensional simple subcoalgebras in Hopf algebras of low dimension in  characteristic 0. We extend their results to $p^2$-dimensional Hopf algebras in characteristic $p>2$. 

\begin{lem}\label{4dimC}
Let $\kk$ be an algebraically closed field of characteristic $p>2$, and $H$ a nonsemisimple  Hopf algebra of dimension $p^2$ over $\kk$ such that $H$ and  $H^*$ are not pointed. Then $H$ and $H^*$ are unimodular, and hence $S^4=\id$. If, in addition, $S^2 \ne \id$, then  $H$ has no $S$-invariant 4-dimensional simple subcoalgebras.
\end{lem}
\begin{proof}
   Since $H$ is not semisimple and $H^*$ is not pointed, $G(H^*)$ is trivial by Theorem \ref{thm:basic}, and hence $H$ is unimodular. We claim that $H^*$ is also unimodular. It is true if $H^*$ is semisimple. On the other hand, if $H^*$ is not semisimple, then $G(H)$ has to be trivial; otherwise, $H$ would be pointed by Theorem \ref{thm:basic}. In particular,  $H^*$ is unimodular. Thus by Radford's $S^4$-formula,  $S^4 =\id$. 
    
    Now suppose $S^2\neq {\rm id}$ and assume that $H$ has a 4-dimensional simple subcoalgebra $C$ such that $S(C)=C$. Let $K$ be the subalgebra of $H$ generated by $C$, which is  a noncocommutative Hopf subalgebra of $H$. Since the $p$-dimensional Hopf algebras over $\kk$ are cocommutative (cf. \cite{NW19} or Theorem \ref{thm:p1}), $\dim K\neq p$. Hence $K=H$ by the Nichols-Zoeller theorem. As a consequence, ${\rm ord}(S^2)={\rm ord}(S^2|_C)=2$.

    Note that the argument of \cite[Thm. 1.4(a)]{St} remains valid when ${\rm char}(\kk)=p\nmid {\rm ord}(S^2)=2$. Hence we can choose a  basis $\{e_{ij}\}_{1\le i,j\le 2}$ of $C$ such that $\Delta(e_{ij})=\sum_{1\le k\le 2} e_{ik}\otimes e_{kj}$, $\epsilon(e_{ij})=\delta_{ij}$ and $S|_{C}$ is given by 
    \[
\begin{pmatrix}
S(e_{11}) & S(e_{12})\\
S(e_{21}) & S(e_{22})
\end{pmatrix}
=
\begin{pmatrix}
e_{22} & -\i e_{12}\\
\i e_{21} & e_{11}
\end{pmatrix},
    \]
    where $\i \in \kk$ satisfies $\i^2=-1$. By the antipode axiom, the following relations hold in $H$:
    \begin{equation}\label{RSL1}
         e_{12}e_{11}=\i e_{11}e_{12},\quad  e_{22}e_{12}=\i e_{12}e_{22}, \quad e_{21}e_{11}=\i e_{11}e_{21},
    \end{equation}
 \begin{equation}\label{RSL2}
         e_{22}e_{21}=\i e_{21}e_{22},\quad e_{12}e_{21}=e_{21}e_{12}, \quad e_{11}e_{22}+e_{12}e_{21}=1,
    \end{equation}
\begin{equation}\label{RSL3}
        e_{11}e_{22}-e_{22}e_{11}=-2\i e_{12}e_{21}.
    \end{equation}
    As a consequence, $H$ is a quotient of the quantum coordinate ring $\mathcal O_q(SL_2)$ over $\kk$ with $q=\i$. Consider the subalgebra $D$ of $H$ generated by $e_{11}^2,e_{12}^2,e_{21}^2,e_{22}^2$. It is straightforward to check that $D$  is a commutative Hopf subalgebra with $\Delta(e_{ij}^2)=\sum_{1\le k\le 2} e_{ik}^2\otimes e_{kj}^2$, $\epsilon(e_{ij}^2)=\delta_{ij}$ and 
\[  \begin{pmatrix}
S(e_{11}^2) & S(e_{12}^2)\\
S(e_{21}^2) & S(e_{22}^2)
\end{pmatrix}
=
\begin{pmatrix}
e_{22}^2 & -e_{12}^2\\
-e_{21}^2 & e_{11}^2
\end{pmatrix}.
\]
Thus $D$ is a quotient of the commutative Hopf algebra $\mathcal O(SL(2))$ over $\kk$. It is clear that $D\neq H$ since $H$ is not involutive. So $\dim D=p$ or $1$. We have two cases to consider according to the classification of $p$-dimensional Hopf algebras over $\kk$ (cf. \cite{NW19} or Theorem \ref{thm:p1}):  (1) $D\cong \kk[x]/(x^p-x)$ with $x$ being primitive. Since $x$ is semisimple primitive, we know $S^2={\rm id}$ by Proposition \ref{prop:semiprim}, a contradiction. (2) $D\cong \kk[x]/(x^p)$ with $x$ being primitive, $D \cong \kk[C_p]$ or $D=\kk$. Then $(D^+)^p=0$ and so $y^p=0$ for $y \in \{e_{11}^2-1,e_{12}^{2},e_{21}^2,e_{22}^2-1\}$. Note that $Hy=yH$ is nilpotent for $y \in \{e_{21}, e_{12}\}$. Thus, $Jac(H) \supseteq (e_{12}, e_{21})$. It follows from the relations \eqref{RSL1} and \eqref{RSL2} that $H/Jac(H)$ is commutative. Thus, every simple $H$-module is 1-dimensional; hence, $H^*$ is pointed, which is a contradiction.
\end{proof}

The following result is simply the dual version of the preceding lemma. 

\begin{lem}\label{4dimR}
Let $\kk$ be an algebraically closed field of characteristic $p>2$, and $H$  a non-cosemisimple Hopf algebra of dimension $p^2$ such that $H$ and $H^*$ are not pointed. Then $H$ and $H^*$ are unimodular and hence $S^4 = \id$. If, in addition, $S^2 \ne \id$, then $H$ has no self-dual 2-dimensional simple modules.
\end{lem}

Now, we prove the main result of this section.

\begin{thm}\label{thm:5}
    Let $\kk$ be an algebraically closed field of positive characteristic $p\le 5$, and $H$ a $p^2$-dimensional Hopf algebra over $\kk$. Then $H$ or $H^*$ is pointed. 
\end{thm}
\begin{proof}
The statement is obvious when $p=2$. Now we consider $p=3$. By \cite[Cor. 3.2(i)]{EG98}, $H$ cannot be both semisimple and cosemisimple since $\dim H=0$ in $\kk$. Without loss of generality, we assume $H$ to be nonsemisimple by duality. Suppose that $H^*$ is not pointed. In particular,  $H$ is not local. Hence $|{\rm Irr}(H)|\ge 3$ by Lemma \ref{noIRR}, and $G(H^*)$ is trivial by Theorem \ref{thm:basic}. Since $\dim H=9$, a simple dimension argument yields that ${\rm Irr}(H)=\{V_0,V_1,V_2\}$ where $\dim V_0=1$ and $\dim V_1=\dim V_2=2$ with $P(V_i)=V_i$ for all $i$, which contradicts to the fact that $H$ is nonsemisimple. Therefore, $H^*$ is pointed.

It remains to deal with the case for $p=5$. We may still assume $H$ to be nonsemisimple. Suppose $H$ and $H^*$ are not pointed. Then $G(H^*)$ is trivial by Theorem \ref{thm:basic} and $S^4=\id$ by Lemma \ref{4dimC}. We  consider the cases (i) $\ord(S^2)=2$ and (ii) $\ord(S^2)=1$, separately.

(i) Suppose $\ord(S^2)=2$. Write ${\rm Irr}(H^*)=\{\kk=V_0,V_1,\ldots, V_n\}$. Since $H$ is not pointed, $n\ge 3$ by Lemma \ref{noIRR} and $|G(H)|=1$ or $5$. The group $G(H)$ can be identified with the 1-dimensional $H^*$-modules $\kk_g$ for $g \in G(H)$.  We first claim that $H^*$ is also nonsemisimple. If not, then
\begin{equation*}
25=1 + \sum_{d_i \ge 2} d_i^2 \quad  \text{ or }\quad 25=5 + \sum_{d_i \ge 2} d_i^2
\end{equation*}
respectively for $|G(H)|=1$ or $5$. By  enumeration, $d_i \in \{2, 4\}$ and  there must be  $i$ such that $d_i =2$. Let $V \in \irr(H^*)$ such that $\dim(V)=2$. Since $H^*$ is semisimple, $[V^* \o V: \kk] =1$, and so
$V^* \o V = \kk \oplus W$ for some $H$-module $W$ of dimension $3$. Therefore, $W$ is not simple and  $[W:\kk_g] \ge 1$ for some nontrivial $g \in G(H)$.  Hence $[V: V \o \kk_g] = 1$ or, equivalently,   $V \o \kk_g \cong V$. Since $g$ generates $G(H)$, $G(H) \le \stab (V)$. By Lemma \ref{l:freeness}, $5  \mid \dim (V)$,  a contradiction! 
Therefore, $H^*$ is also nonsemisimple. Since $H$ is not pointed, $G(H)$ is also trivial by Theorem \ref{thm:basic}.

We next claim $H$ must have a nonzero primitive element. Suppose not. By Lemma \ref{lem:dext}, $\Ext_{H^*}^1(\kk,\kk)=0$. Then $[P(V_0): V_i] > 0$ for some $i >0$, say $i=1$.  If $\dim V_1 =2$, we may assume $V_2 = V_1^*$ since $H^*$ has no 2-dimensional self-dual simple module by Lemma \ref{4dimR}. Since $P(V_0)$ is also self-dual,  $[P(V_0): V_2] > 0$  and so $\dim P(V_0) \ge 6$. In this case, $[P(V_1): V_0], [P(V_2): V_0] \ge 1$. Therefore, 
$\dim P(V_1), \dim P(V_2) \ge 5$, and hence
$$
25 \ge  \dim P(V_0)+\sum_{i=1}^2 \dim (V_i) \dim P(V_i)  \ge 6 + 2 \cdot 5 \cdot 2 = 26,
$$
a contradiction! Therefore, $\dim V_1 \ge 3$. This implies that $\dim P(V_0) \ge 5$ and $\dim P(V_1) \ge 7$. Thus,
$$
25 \ge \dim P(V_0)+\dim (V_1) \dim P(V_1) = 5+3 \cdot 7 =26, \text{ a contradiction again!} 
$$
Therefore, $H$ must have a nonzero primitive element.

Let $x \in H$ be a nonzero primitive element. Then $\kk[x]$ is commutative Hopf subalgebra of $H$. Since $H^*$ is not pointed, $\kk[x]$ is of dimension $p$. It follows from Corollary \ref{cor:cad} that $\dim H_0 \le 5$. Since $\dim(V_i) \ge 2$ for $i\ge 1$ and $n\ge 3$,
$$
5 \ge \dim (H_0)=1 +\sum_{i=1}^n \dim(V_i)^2 \ge 1+ 3 \cdot 2^2 = 13,
$$
a contradiction! Therefore, $H$ or $H^*$ must be pointed if $\ord(S^2)=2$.

(ii) Suppose $S^2={\rm id}$. Write ${\rm Irr}(H)=\{\kk=V_0,V_1,\ldots,V_n\}$. Since $H^*$ is not pointed,  $\dim V_i\ge 2$ for $i > 0$ and $n \ge 3$ by Lemma \ref{noIRR}. Since $H$ is not semisimple, $\sum_{(\Lambda)} S(\Lambda_2)S^2(\Lambda_1)=\e(\Lambda)1=0$, where $\Lambda$ is a nontrivial left integral of $H$. By \cite[Lem. 3(b)]{RS}, $\tr(r(a) \circ S^2) =\tr(r(a))=0$ for all $a\in H$, where $r(a)$ denotes the operator on $H$ given by the right multiplication of $a$. As a consequence, for any indecomposable projective module $He$ for some primitive idempotent $e\in H$, we have $\dim He=\tr(r(e))=0$ in $\kk$, and hence $p \mid \dim He$. In particular, $\dim P(V_i) \ge p$. Now, we have
$$
p^2 = \sum_{i=0}^n \dim V_i \cdot \dim P(V_i) \ge (1+2n)p
$$
and so $n \le \frac{p-1}{2} = 2$, a contradiction.
Therefore, $H^*$ must be pointed if $S^2=\id$.
\end{proof}

According to \cite{WW14} or Theorem \ref{thm:p1}, the  $p^2$-dimensional pointed Hopf algebras are closed under duality, which is stated in Corollary \ref{c:dual_closure}. The following theorem is an immediate consequence of Theorem \ref{thm:5} and Corollary \ref{c:dual_closure}, which provides affirmative answers to both Questions A and B for $p \le 5$.
\begin{thm}
    Let $\kk$ be an algebraically closed field of positive characteristic $p \le 5$, and $H$ a $p^2$-dimensional Hopf algebra over $\kk$. Then $H$ is pointed.
\end{thm}

\section{Twist equivalence of cleft extensions}
In this section, we introduce the notion of twist equivalence for cleft extensions and crossed products. As an application, we classify certain cleft extensions, up to twist equivalence, over group algebras of elementary abelian $p$-groups and their duals over an algebraically closed field $\kk$ of characteristic $p$.  
The readers are referred to \cite[Ch. 7]{Mont93} for definitions and results on cleft extensions.

Let $B$ be a right $H$-comodule algebra and $A$ a subalgebra of $B$. The extension $A \subseteq B$ is called an {\it $H$-cleft extension} if $B^{co H}=A$ and there exists a right $H$-comodule map $\gamma: H\to B$ which is convolution invertible, and $\gamma$ is called a \emph{cleft map}.  By \cite[Thm. 2.2]{Sch}, if $\pi: A \to  H$ is a surjection of Hopf algebras,  then $A^{co \pi} \subset A$ is an $H$-cleft extension. In particular, for the short exact sequence \eqref{E:SES}, $A$ is an $H$-cleft extension of $\iota(K)$. 

According to \cite[Thm. 7.2.2]{Mont93}, $H$-cleft extensions can be characterized by crossed products, which we now recall.  Let $H$ be a Hopf algebra that {\it measures} an algebra $A$, namely there is a $\kk$-linear map $H\otimes A\to A$, given by $h\otimes a \mapsto h\cdot a$ such that $h\cdot 1=\epsilon(h)1$ and $h\cdot (ab)=\sum_{(h)} (h_1\cdot a)(h_2\cdot b)$, for all $h\in H$ and $a,b\in A$. In this paper, we simply call the triple $(H,A, \cdot)$ a \emph{measurement}. It is worth noting that $A$ may not be an $H$-module algebra as the $H$-action on $A$ may not be associative. 

Given a measurement  $(H,A,\cdot)$ and a convolution invertible map $\sigma \in \Hom_\kk(H \o H, A)$, one can construct the {\it crossed product} $A\#_\sigma H$ of $A$ with $H$, which is the vector space $A\otimes H$ equipped with multiplication given by 
 \begin{align}\label{mult}
     (a\#h)(b\#k)~=~\sum_{(h),(k)} a(h_1\cdot b)\sigma(h_2,k_1)\# h_3k_2
 \end{align}
 for all $h,k\in H$ and $a,b\in A$. In this case, the associativity of \eqref{mult} with identity element $1_A\#1_H$ is equivalent to the fact that $A$ is a {\it twisted} $H$-module and $\sigma$ is a 2-cocycle (see \cite[Lem. 7.1.2]{Mont93} and the definitions therein) of the measurement $(H,A, \cdot)$. Note that the crossed product $A\#_\sigma H$ canonically admits a left $A$-module and a right $H$-comodule structure. In particular, $A\#_\sigma H$ is an $H$-cleft extension of $A$  with the cleft map $\gamma: H \to A\#_\sigma H,\, h \mapsto 1_A\# h$, under the identification $A \equiv A\# 1_H$ (cf. \cite[Prop. 7.2.7]{Mont93}).

Recall that one can twist a left $A$-module $M$ by an algebra automorphism $\varphi$ of $A$ to $_\varphi \!M$. By duality, if $\phi$ is a coalgebra automorphism of a coalgebra $C$, and $V$ is a right $C$-comodule with the $C$-coaction $\rho: V \to V \o C$, $V^\phi$ denotes the right $C$-comodule with the $C$-coaction given by $(\id_V \o \phi)\circ\rho$. 
\begin{deff}
    Let $H$ be a Hopf algebra and  $A$ an  algebra.  
    \begin{enumerate}[label=\rm{(\roman*)}]
        \item   We say  two $H$-cleft extensions $B$, $B'$ of $A$  are \emph{twist equivalent}, if there exist an algebra automorphism $\varphi$ of $A$,   a Hopf algebra automorphism $\phi$ of $H$ and 
          an algebra isomorphism $f: B\to \!_\varphi (B')^{\phi}$ that is a left $A$-module and a right $H$-comodule map. In this case, $f$ is called a ($\phi, \varphi$)-\emph{twist isomorphism} or simply a \emph{twist isomorphism} of the $H$-cleft extensions.
        \item  Two crossed products $A\#_\sigma H$ and $A\#'_{\sigma'}  H$ are said to be \emph{twist equivalent} if they are twist equivalent as $H$-cleft extensions of $A$.
         \item  In particular, if $\varphi=\id_A$ and $\phi=\id_H$, we will say an {\it equivalent} (resp. {\it isomorphism}) instead of twist equivalent (resp. \emph{twist isomorphism}) of $H$-cleft extensions or crossed products as in (i) and (ii).
    \end{enumerate}     
\end{deff}
 One can easily see that twist equivalence is an equivalence relation on the set of all $H$-cleft extensions of $A$.

The following result is a natural extension of \cite[Thm. 7.3.4]{Mont93} to the context of twist equivalence. We include its proof for the sake of completeness.  
\begin{prop}\label{twistedE}
  Let $(H, A, \cdot)$ and $(H, A, \bullet)$ be measurements, $\sigma,\sigma'$ be 2-cocycles of the measurements respectively, and $A\#_\sigma H$, $A\#'_{\sigma'} H$  their corresponding crossed products. Suppose 
    \[
   \Phi:~A\#_\sigma H~\to~_\varphi(A\#'_{\sigma'} H)^{\phi}
    \]
    is a $(\phi, \varphi)$-twist isomorphism for some algebra automorphism $\varphi$ of $A$ and some Hopf automorphism $\phi$ of $H$. Then there exists a convolution invertible map $u\in \Hom_\kk(H,A)$ satisfying $u(1_H)=1_A$  with the convolution inverse $\ol u$ such that for all $a\in A$, $h,k\in H$,
    \begin{itemize}
\item[\rm(i)] $\Phi(a\#h)=\sum_{(h)} \varphi(au(h_1))\#' \phi^{-1}(h_2)$,
\item[\rm (ii)] $h\bullet a=\varphi\left(\sum_{(h)} \ol u(\phi(h_1))\,\left(\phi(h_2)\cdot \varphi^{-1}(a)\right)\, u(\phi(h_3))\right)$,
\item[\rm (iii)] $\sigma'(h,k)= m \left( \ol u (\phi(h_1)) \# \phi(h_2))(\ol u (\phi(k_1)) \# \phi(k_2))\right)$ where the $\kk$-linear map $m: A \#_\sigma H \to A$ is given by $m(a\# h) = \varphi(au(h))$.
\end{itemize}
    Conversely given a convolution invertible map $u\in \Hom_\kk(H,A)$ with $u(1_H)=1_A$ such that {\rm(ii)} and {\rm (iii)} hold, then the map $\Phi$ in {\rm (i)} is a $(\varphi,\phi)$-twist isomorphism.  
\end{prop}
\begin{proof}
We define $u\in \Hom_\kk(H,A)$ by $u(h)=(\varphi^{-1}\otimes \epsilon)\Phi(1_A\# h)$ for $h\in H$. Then  $\varphi(u(h))=(\id\otimes \epsilon)\Phi(1_A\# h)$ and $\Phi(a\# h)=\varphi(a)\Phi(1_A\#h)$ as $\Phi$ is a left $A$-module map. Since $\Phi$ is a right  $H$-comodule map, we have 
$$
(\id\otimes \Delta)\circ \Phi=(\Phi \otimes \phi^{-1})\circ (\id \otimes \Delta)\,.
$$
By applying $(\id \o m_H)\circ (\id \otimes \epsilon 1_H \otimes \id)$ to both sides of this equation, we get
\begin{align*}
\Phi(1\#h)&=\,(\id \o m_H)\circ (\id \otimes \epsilon 1_H \otimes \id)\left(\sum \Phi(1_A\# h_1) \o \phi^{-1}(h_2)\right)\\
&=\,\sum (\id \o m_H) (\varphi(u(h_1))\#'1_H \o \phi^{-1}(h_2))\\
&=\, \sum \varphi(u(h_1))\#' \phi^{-1}(h_2),
\end{align*}
which proves the formula in (i). Note that $\Phi(1_A\#1_H)=1_A\#1_H$ implies that $u(1_H)=1_A$.

It is immediate to see that  $\Phi^{-1}: A\#'_{\sigma'} H\to {}_{\varphi^{-1}}(A\#_\sigma H)^{\phi^{-1}}$ is a $(\varphi^{-1},\phi^{-1})$-twist isomorphism. By (i), there exists  $v\in \Hom_\kk(H,A)$ given by $v(h)=(\varphi\otimes \epsilon)\Phi^{-1}(1\#' h)$ such that $\Phi^{-1}(a\#' h)=\sum \varphi^{-1}(av(h_1))\# \phi(h_2)$. We claim that $\ol u=\varphi^{-1}\circ v\circ \phi$ is the convolution inverse of $u$. This is because
\begin{align*}
1_A\#h=\Phi^{-1}\Phi(1_A\#h)& =\Phi^{-1}\left(\sum \varphi(u(h_1))\#' \phi^{-1}(h_2)\right)\\
& =\sum u(h_1)\varphi^{-1}(v(\phi^{-1}(h_2)))\# h_3.
\end{align*}
By applying $\id_A \otimes \epsilon$ to this equation, we see that $u*(\varphi^{-1}\circ v\circ \phi^{-1})=\epsilon 1_A$ in $\Hom_\kk(H,A)$. Similarly, we see $v*(\varphi\circ u\circ \phi)=\epsilon 1_A$, which is equivalent to $(\varphi^{-1}\circ v \circ \phi^{-1})* u=\epsilon 1_A$. Thus,  $\ol u=\varphi^{-1}\circ v\circ \phi^{-1}$ is the convolution inverse of $u$. 

Now we consider the equation $\Phi^{-1}((1_A\#' h)(1_A\#'k))=\Phi^{-1}(1_A\#' h)\Phi^{-1}(1_A\#'k)$. We apply \eqref{mult} to obtain  
\begin{align*}
&\sum\varphi^{-1}\left(\sigma'(h_1,k_1)v(h_2k_2)\right)\# \phi(h_3k_3)\\
=&\sum \varphi^{-1}(v(h_1))\left(\phi(h_2)\cdot \varphi^{-1}(v(k_1))\right)\sigma(\phi(h_3),\phi(k_2))\# \phi(h_4k_3).
\end{align*}
This proves (iii) after applying $m$ to the equation.

Applying $\id\otimes \epsilon$ to  $\Phi^{-1}((1_A\#' h)(b\#'1_H))=\Phi^{-1}(1_A\#' h)\Phi^{-1}(b\#1_H)$ or
$$
\sum\varphi^{-1}\left((h_1\bullet b)v(h_2)\right)\# \phi(h_3)
=\sum \varphi^{-1}(v(h_1))\left(\phi(h_2)\cdot \varphi^{-1}(b)\right)\# \phi(h_3),
$$
we find
\begin{align*}
\sum \varphi^{-1}(h_1\bullet b)\varphi^{-1}(v(h_2))=\varphi^{-1}(v(h_1))\left(\phi(h_2)\cdot \varphi^{-1}(b)\right).
\end{align*}
Note that $\ol u\circ \phi=\varphi^{-1}\circ v$ and $\varphi^{-1}\circ v$ is the convolution inverse of $u \circ \phi$. Inverting $\varphi^{-1}\circ v$ gives (ii). 

Finally, the converse follows by direct computation. 
\end{proof}

Throughout the remainder of this section, we will assume that $H$ is a cocommutative Hopf algebra and $A$ is a commutative algebra. In this setting, a crossed product $A\#_\sigma H$ is defined as an $H$-module algebra $A$ equipped with a 2-cocycle $\sigma: H \otimes H \rightarrow A$, which can be placed within the context of Sweedler's cohomology $H^\bullet_{sw}(H,A)$, as described in \cite{Sw68}. In particular, the measurement $(H,A, \cdot)$ simply means an $H$-module algebra. 

For any integer $n\ge 0$, denote by ${\rm Reg}^n(H,A)\subseteq \Hom_\kk(H^{\otimes n},A)$ the set of all convolution invertible linear maps, which forms an abelian subgroup under the convolution product.  The coface operator $d^i:\Reg^n(H,A) \to \Reg^n(H,A)$ is defined  as follows:
\[
d^i(f)(h_0\o \cdots\o h_n)=
\begin{cases}
h_0f(h_1\o\cdots\o h_n) & \text{for}\ i=0,\\
f(h_0\o \cdots\o h_{i-1}h_i\o \cdots\o h_n) & \text{for}\ 1\le i\le n,\\
f(h_0\o \cdots\o h_{n-1})\epsilon(h_n) & \text{for}\ i=n+1
\end{cases}
\]
for any $f\in {\rm Reg}^n(H,A)$ and $h_0, \dots, h_n \in H$. This gives rise to a cochain complex $({\rm Reg}^\bullet(H,A),\partial)$ whose differential $\partial$ is defined by 
\[
\partial(f)=d^0(f)*d^1(f^{-1})*\cdots *d^{n+1}(f^{\pm 1})
\]
for any $f\in {\rm Reg}^n(H,A)$ where $*$ denotes the convolution product and $f^{-1}$ means the convolution inverse of $f$. The {\it Sweedler cohomology $H_{sw}^\bullet(H,A)$ of $H$ with coefficients in $A$} is defined to be the cohomology of the above cochain complex. In particular,
\[
H_{sw}^n(H,A):={\rm Ker}\, \partial^n/{\rm Im}\, \partial^{n-1},
\]
for any $n\ge 0$ under the convention $\partial^{-1}$ is trivial.

As noted in \cite{Sw68}, one can use the normalized  subcomplex $({\rm Reg}_+^\bullet(H,A),\partial)$ of $({\rm Reg}^\bullet(H,A),\partial)$ to compute $H_{sw}^\bullet(H,A)$, where
${\rm Reg}_+^n(H,A)$ consists of those convolution invertible  maps $f\in \Hom_\kk(H^{\o n}, A)$ satisfying $f(x_1 \o \cdots \o x_n) = \e(x_1 \cdots x_n) 1_A$ whenever $1_H \in \{x_1, \dots, x_n\}$.
 In particular, $\sigma\in {\rm Reg}_+^2(H,A)$ is a 2-cocycle if and only if $d^0(\sigma) * d^1(\sigma^{-1})*d^2(\sigma) *d^3(\sigma^{-1})=\epsilon\otimes \epsilon \o \e \, 1_A$. Since  ${\rm Reg}_+^2(H,A)$ is an abelian group, $\sigma\in {\rm Reg}_+^2(H,A)$ is a 2-cocycle if and only if
\begin{align}\label{eq:2cocycle}
\sum_{(h),\, (k),\, (m)} h_1\cdot\sigma(k_1,m_1)\,\sigma(h_2,k_2m_2)=\sum_{(h),\, (k)} \sigma(h_1,k_1)\,\sigma(h_2k_2,m)
\end{align}
for any $h,k,m\in H$.  

The first assertion of our following result is well-known since Sweedler's second cohomology classifies all cleft extensions up to equivalence (see \cite[Thm. 8.6]{Sw68}).  

\begin{prop}\label{lem:autS}
       Let $H$ be a cocommutative Hopf algebra and $A$ a commutative $H$-mdoule algebra. 
       \begin{enumerate}
           \item[\rm (i)] If $H_{sw}^2(H,A)$ is trivial, then the crossed product $A\#_\sigma H$ for any Sweedler 2-cocycle $\sigma$ is equivalent to the smash product $A\#H$. 
           \item[\rm (ii)] Suppose $H_{sw}^2(H,A)$ is trivial for any $H$-module algebra structure on $A$. Then every right $H$-cleft extension of $A$ is equivalent to a smash product $A \# H$, where the $H$-module action on $A$ is determined by a cleft map.
           \item[\rm (iii)] Suppose $H_{sw}^2(H,A)$ is trivial for any $H$-module algebra structure on $A$. For any two $H$-actions $\cdot, \bullet: H\otimes A\to A$, the corresponding smash products $A\#H$ and $A\#' H$ are twist equivalent if and only if there is some algebra automorphism $\varphi$ of $A$ and some Hopf automorphism $\phi$ of $H$ such that
       \begin{equation}\label{action}
             h\bullet  a~=~\varphi\left(\phi(h)\cdot \varphi^{-1}(a)\right)\,  
       \end{equation}    
       for all $h\in H$ and $a\in A$.
       \end{enumerate}
\end{prop}
\begin{proof} 
(i) Since $H^2_{sw}(H,A)$ is trivial, the Sweedler 2-cocycle $\sigma: H\otimes H\to A$ is a 2-coboundary, namely there is some $u\in \Reg^1_+(H,A)$ such that $u(1_H)=1_A$ and
\begin{align*}
\sigma(h,k)~=~\sum \left(h_1\cdot u(k_1) \right)\, \ol u(h_2k_2) \, u(h_3),   
\end{align*}
where $\ol u$ is the convolution inverse of $u$. Then by applying Proposition \ref{twistedE} in case of $\varphi,\phi=\id$, one can check that $f: A\#_\sigma H \to A\# H$ defined by $f(a\# h)=\sum_{(h)} au(h_1)\#h_2$  with inverse $f^{-1}(a\#h)= \sum_{(h)} a\ol u(h_1)\#h_2$ for any $a\in A$ and $h\in H$ is an isomorphism of $H$-cleft extensions.

(ii) By \cite[Thm. 7.2.2]{Mont93}, every $H$-cleft extension of $A$ is equivalent to $A \#_\s H$ as $H$-cleft extensions of $A$ for some Sweedler's 2-cocycle $\s \in \Reg^2_+(H,A)$ of the $H$-module algebra $A$ with its $H$-action determined by a cleft map.  Since $H^2_{sw}(H,A)$ is trivial, the statement follows from (i).

(iii) Suppose $\Phi: A\#H\to \!_\varphi(A\#' H)^\phi$ is a twist isomorphism of smash products for some algebra automorphism $\varphi$ of $A$ and some Hopf algebra automorphism $\phi$ of $H$. By Theorem \ref{twistedE}, there is some $u\in \Reg^1_+(H,A)$ with $u(1_H)=1_A$ and convolution inverse $\ol u$ such that 
$$
    h\bullet a=\varphi\left(\sum_{(h)} \ol u(\phi(h_1))\,\left(\phi(h_2)\cdot \varphi^{-1}(a)\right)\, u(\phi(h_3))\right)
    =\varphi\left(\phi(h)\cdot \varphi^{-1}(a)\right),
$$
where the last equality follows from the cocommutativity of $H$ and the commutativity of $A$. Moreover, by letting $u=\epsilon 1_A$ and $\sigma=\sigma'=\epsilon 1_A\otimes \epsilon 1_A$, one can check (iii) in Theorem \ref{twistedE} holds. This implies our result. 
\end{proof}

It was observed in \cite[Thm. 5.10]{Gui} and \cite[Prop. 5.7]{EG21} that  $H^2_{sw}(\kk[C_p^n]^*,\kk)$ is trivial for any positive integer $n$. Our next result extends this vanishing property to nontrivial coefficients. 

\begin{thm}\label{thm:zero2nd}
Let $\kk$ be an algebraically closed field of characteristic $p>0$, and $K=\kk [C_p^n]^*$ for some positive integer $n$.  For any finite-dimensional commutative $K$-module algebra $A$ over $\kk$,  $H^2_{sw}(K,A)$ is trivial. 
\end{thm}

\begin{proof}
Note that $\kk[C_p^n]^*$ is a connected semisimple Hopf algebra. By the Hochschild theorem \cite{Ho54}, we have $\kk[C_p^n]^* \cong \kk[x]/(x^{p^n}-x)$ as Hopf algebras, where $x$ is primitive. Now, we simply assume $K=\kk[x]/(x^{p^n}-x)$ with $x$ primitive. Suppose $\sigma\in {\rm Reg}_+^2(K,A)$. Then $\sigma(1\o 1)=1$ and $\sigma(x^i\o  1) = \sigma(1\o  x^i)=0$ for $i \ge 1$. We claim that there exists $f\in {\rm Reg}_+^1(K,A)$ such that 
\begin{align}\label{eq:cond1}
    (\sigma*\partial f^{-1})(x^i\o x)=0,
\end{align}
for $0\le i\le p^n-1$. We proceed with the proof of the above claim in several steps.

(i) We first find $f\in {\rm Reg}_+^1(K,A)$, where \eqref{eq:cond1} holds for $0\le i\le p^n-2$.  Note that \eqref{eq:cond1} holds trivially for $i=0$ and one can rewrite \eqref{eq:cond1} as $(\sigma * d^1f)(x^i\o x)=(d^0f*d^2f)(x^i\o x)$ or 
\begin{equation}\label{eq:cob1}
f(x^{i+1})=-\sum_{j=1}^i{i\choose j}\,\sigma(x^j\o x)f(x^{i-j})+\sum_{j=0}^i\, {i\choose j}\, (x^j\cdot f(x))f(x^{i-j})
\end{equation}
for $1\le i\le p^n-2$. For any $a \in A$, we can define $f(1)=1$, $f(x)=a$, and  set $f(x^{i+1})$ inductively by \eqref{eq:cob1} for $1\leq i\leq p^n-2$. Then \eqref{eq:cond1} holds for $0\le i\le p^n-2$.

(ii) By (i), we may assume $\sigma(x^i\o x)=0$ for all $0\le i\le p^n-2$. Then we set $h=x,k=x^{p^n-1},m=x$ in \eqref{eq:2cocycle} to get
\begin{align*}
\sum_{i,l =0}^1 \sum_{j=0}^{p^n-1} \binom{p^n-1}{j}(x^i\cdot \sigma(x^j, x^l)) \sigma(x^{1-i}, x^{p^n-j-l})\\
=\sum_{i =0}^1 \sum_{j=0}^{p^n-1} \binom{p^n-1}{j} \sigma(x^i, x^j) \sigma(x^{p^n-i-j}, x)\,,
\end{align*}
which is equivalent to 
\[
x\cdot \sigma(x^{p^n-1}, x)+\sigma(x,x) = \sigma(x,x)\,.
\]
Therefore, $x\cdot \sigma(x^{p^n-1},x)=0$, namely, $\sigma(x^{p^n-1},x)\in A^K$. 

(iii) Now for any $f\in {\rm Reg}_+^1(K,A^K)$ that satisfies \eqref{eq:cond1} for $0\leq i\leq p^n-2$, \eqref{eq:cob1} yields $f(x^{i+1})=f(x)^{i+1}$ for $0\le i\le p^n-2$ and when $i=p^n-1$ it is equal to 
\[
f(x)=f(x^{p^n})=-\sigma(x^{p^n-1}\o x)+f(x)f(x^{p^n-1})=-\sigma(x^{p^n-1}\o x)+f(x)^{p^n}.
\]

(iv) We show that the polynomial equation $F(X) = 0$, where $F(X)=X^{p^n}-X-\sigma(x^{p^n-1},x) \in A^K[X]$ has a solution in $A^K$. Since $A^K$ is finite-dimensional commutative over an algebraically closed field $\kk$, the equation $F(X)$ has solution in $A^K/J(A^K)=\kk\times \cdots \times \kk$, where $J(A^K)$ denotes the Jacobson radical of $A^K$. Note $A^K$ is complete with respect to its $J(A^K)$-adic filtration.  Since $F'(X)=-1 \ne 0$ in $A^K$, we conclude that $F(X)=0$ has a solution in $A^K$ by Hensel's lifting lemma. We choose $\alpha\in A^K$ as a solution that $F(\alpha)=0$. 

(v) We construct $f\in {\rm Reg}_+^1(K,A^K)$ by setting $f(1)=1$, $f(x)=\alpha$ and $f(x^{i+1}) = f(x)^{i+1}$ for $1\leq i\leq p^n-2$. By the previous discussion, it is straightforward to check that $(\sigma * \partial f^{-1})(x^i,x)=0$ for all $0\le i\le p^n-1$. This proves our claim. 

Now up to a coboundary, we can assume that  $\sigma(x^t,x)=0$ for all nonnegative integers $t$. By setting $h=x,k=x^t,m=x$ in \eqref{eq:2cocycle} again, we find
\begin{align*}
\sum_{i,l =0}^1 \sum_{j=0}^{t} \binom{t}{j}(x^i\cdot \sigma(x^j, x^l)) \sigma(x^{1-i}, x^{t-j-l+1})\\
=\sum_{i =0}^1 \sum_{j=0}^{t} \binom{t}{j}\sigma(x^i, x^j) \sigma(x^{t-i-j+1}, x)\,,
\end{align*}
which implies
that  $\sigma(x,x^{t+1})=0$. Therefore, $\sigma(x, m)=0$ for all $m \in K$. Inductively, suppose $t$ is a positive integer such that $\sigma(x^j,m)=0$  for all $m \in K$ whenever $1 \le j \le t$.  By setting $h=x,k=x^t$ in \eqref{eq:2cocycle}, we find
\[
x \cdot \sigma(x^t, m) = \sigma(x^{t+1}, m) 
\]
for all $m \in K$. Therefore, $\sigma(x^{t+1}, m)=0$ for all $m \in K$. Thus, $\sigma$ is trivial. This completes our proof.   
\end{proof}

\begin{remark}
   The assumption that $\kk$ is algebraically closed is essential for Theorem  \ref{thm:zero2nd}. For instance, if $\kk$ is of characteristic $p=2$ (not necessarily algebraically closed), then \cite[Thm. 1.1]{Gui} states that $H^2_{sw}(\kk[C_p^n]^*,\kk)=(\kk/\{x+x^2\,|\, x\in \kk\})^{\oplus n}$.
\end{remark}

An immediate application of Theorem \ref{thm:zero2nd} is the classification of certain cleft extensions up to twist equivalence over $\kk[C_p^n]$ and their duals.

\begin{prop}\label{CleftCp}
Let $\kk$ be an algebraically closed field of characteristic $p>0$, and $H=\kk [C_p^n]^*$ for some positive integer $n$. Then any $H$-cleft extension of a finite-dimensional commutative local algebra $A$ over $\kk$ is twist equivalent to some smash product $A\# H$. Moreover, there are exactly 3 twist inequivalent smash products $A\# H$, $A\#' H$, and  $A\#''H$, and their $H$-actions on $A$ are respectively given by  \[
x\cdot t~=~0,\ t,\ t+1,
\]
where $x$ is a primitive generator of $H$ satisfying $x^{p^n}=x$, and $t$ is a nilpotent generator of $A$. The last action $x\cdot t=t+1$ only occurs when $p\mid \dim A$. 
\end{prop}
\begin{proof}
By Theorem \ref{thm:zero2nd}, $H_{sw}^2(H,A)$ is trivial for any $H$-action on $A$. So it suffices to classify all possible $H$-actions on $A$, which yields twist inequivalent smash products $A\# H$ by Proposition \ref{lem:autS}.  Again, by Hochschild's theorem,  we may assume $H=\kk[x]$ with some primitive generator $x$ satisfying $x^{p^n}=x$. Thus any $H$-action on $A$ is determined by the action of $x$ on $A$ as a $\kk$-derivation satisfying $x^{p^n}=x$. Denote by $\mathbb F_q$ the finite field of  $q=p^n$ elements. Since the operator $x$ is semisimple, there is a decomposition $A=\bigoplus_{i\in \mathbb F_q} A_i$ where $A_i$ is the eigenspace of $x$ for the eigenvalue $i \in \mathbb F_q$. It is immediate to see $A_iA_j\subseteq A_{i+j}$ for  $i, j \in \mathbb F_q$.

Since $A$ is a finite-dimensional local commutative algebra over $\kk$,   $A = \kk[t]/(t^m)$ with $m=\dim A$ and $t$ a nilpotent generator of $A$. Then $t=\sum_{i\in \mathbb F_q} t_i$ where $t_i\in A_i$, and there exists $s \in \mathbb F_q$ such that $t_s =\sum_{i=0}^{m-1} a_i t^i$ with $a_i \in \kk$ and  $a_1 \ne 0$. Thus, $t_s = a_0+tu$ for some invertible element $u \in A$, and hence  $t_s^i \in t^i u^i + F_{i-1}$ where $F_\ell$ is the $\kk$-linear span of $\{t^0, \dots, t^\ell\}$. This implies 
$\{t_s^0, \ldots, t_s^{m-1}\}$ is a $\kk$-basis for $A$ and $1, t_s$ generate $A$ as $\kk$-algebra. 

If $s=0$, then $A = A_0$ and so  $x\cdot t=0$. Now, we assume $s \ne 0$.  Recall that $t_s=a_0+tu$ for some invertible element $u \in A$ and $x\cdot t_s=st_s$. We first consider the case for $a_0\neq 0$ and set $t'=tu/a_0$.  Then  $x\cdot t'=s(t'+1)$. Consider the algebra automorphism $\varphi$ of $A$ defined by $\varphi(t)=t'$ and the Hopf algebra automorphism $\phi$ of $H$ defined by $\phi(x)=x/s$. Since $\varphi^{-1}(\phi(x)\cdot \varphi(t))=t+1$, the smash product in this is twist equivalent to $A\#'' H$, with the $H$-action given by $x\cdot t=t+1$ by \eqref{action}. Since $x\cdot (t^m)=mt^{m-1}(x\cdot t)$, it is easy to see that the action $x\cdot t=t+1$ is not well-defined if $p\nmid m$. 

Now we deal with the case for $a_0=0$. By considering the algebra automorphism $\varphi$ of $A$ defined by $\varphi(t)=tu$ and the Hopf algebra automorphism $\phi$ of $H$ defined by $\phi(x)=x/s$, we can assume that $x\cdot t=t$ up to twist equivalence of smash products by \eqref{action} again.  

  The smash product $A\# H$ given by the trivial $H$-action is commutative, and $A\#'H \cong \kk\langle x,t\rangle/(x^{p^n}-x,t^m,[x,t]=t)$ and $A\#''H \cong \kk\langle x,t\rangle/(x^{p^n}-x,t^m,[x,t]=t+1)$ as algebras. But the former has $p^n$ 1-dimensional characters while the latter has none. Therefore, they are twist inequivalent.  
\end{proof}

Regarding the second $H$-action described in the preceding proposition, when $H=\kk[C_p]^*$ and $A=H^*=\kk[C_p]$, the given smash product is the Heisenberg double of the group algebra $\kk[C_p]$, as shown in the following example.
\begin{example}\cite[Ex. 4.1.10]{Mont93}\label{ExH}
Let $H$ be any finite-dimensional Hopf algebra. There is a left $H$-action $\rightharpoonup$ on the dual Hopf algebra $H^*$  given by
\begin{equation*}
    h\rightharpoonup f~=~\sum_{(h)} f_2(h)f_1
\end{equation*}
for all $h\in H,f\in H^*$. Under this $H$-action, $H^*$ is a left $H$-module algebra. The corresponding smash product $H^*\#H$ is the {\it Heisenberg double} of $H^*$, denoted by $\mathcal H(H^*)$. Note that $\mathcal H(H^*)$ is isomorphic to ${\rm End}_\kk(H^*)$ as algebras (cf. \cite[9.4.3]{Mont93}) and it is an $H$-cleft extension of $H^*$.
\end{example}

\begin{cor}\label{cor:Rp}\label{ExDD}
  Let $\kk$ be an algebraically closed field $\kk$ of characteristic $p>0$ and $H=\kk[C_p]^*$. Then any $H$-cleft extension of the group algebra $\kk[C_p]$ is twist equivalent to one of the following smash products:
 \begin{enumerate}[label=\rm{(\arabic*)}]
     \item The tensor product algebra: $\kk[C_p]\otimes \kk[C_p]^*$.
     \item The Heisenberg double: $\mathcal H(\kk[C_p])$.
     \item The smash product: $\kk[C_p]\# \kk[C_p]^*$, where the $\kk[C_p]^*$-action on $\kk[C_p]$ is given by  $x\cdot g=g-1$ with $C_p=\langle g\rangle$ and $x\in \kk[C_p]^*$ a primitive generator.
 \end{enumerate}   
\end{cor}
\begin{proof} 
Note that $H \cong \kk[x]/(x^p-x)$ as Hopf algebras, where $x$ is primitive.
Our result follows from Theorem \ref{CleftCp} since $\kk[C_p]\cong \kk[t]/(t^p)$ as algebras with $t=g-1$. Case (1) corresponds to the trivial action $x\cdot t=0$. For Case (2), the $H$-action on $\kk[C_p]$ is given by $x\cdot t=t+1$ or $x\cdot g=g$, which coincides with the natural action $\rightharpoonup$ of $\kk[C_p]^*$ on $\kk[C_p]$. By Example \ref{ExH}, the corresponding smash product $\kk[C_p]\# \kk[C_p]^*$ is the Heisenberg double $\mathcal H(\kk[C_p])$. Case (3) corresponds to the action $x\cdot t=t$ or $x\cdot g=g-1$.  
\end{proof}

\begin{prop}\label{ExtG}
  Let $\kk$ be an algebraically closed field of characteristic $p>0$, and $H=\kk[G]$ for some elementary abelian $p$-group $G$. Then any $H$-cleft extension of the semisimple algebra $A=\kk[x]/(x^p-x)$ is twist equivalent to either $A\otimes H$ or $\mathcal H(\kk[C]^*)\otimes \kk[G']$ for some subgroups $C, G'$ of $G$ such that $C \cong C_p$ and $C \oplus G'=G$, where $\kk[G]$ is naturally identified with $\kk[C]\o \kk[G']$ and $A$ is identified with $\kk[C]^*$. In particular, if $G=C_p$, then any $H$-cleft extension of $A=H^*$ is twist equivalent to one of the smash products below: 
 \begin{enumerate}[label=\rm{(\arabic*)}]
     \item The tensor product algebra: $\kk[C_p]^*\otimes \kk[C_p]$.
     \item The Heisenberg double: $\mathcal H(\kk[C_p]^*)$.
 \end{enumerate} 
\end{prop}
\begin{proof}
We write $A=\bigoplus_{i\in \Z_p} \kk\, e_i$ where $\{e_0,\ldots,e_{p-1}\}$ is a complete set of primitive orthogonal idempotents of $A$ and $x=\sum_{i\in \Z_p} ie_i$.  By \cite[Thm. 3.1]{Sw68}, we have 
$$
H_{sw}^2(\kk[G],A)=H^2(G,A^\times).
$$
It follows from Lemma \ref{l:vanishing_coh} that $H_{sw}^2(\kk[G],A)$ is trivial for any $G$-action on $A$.

 Now, it suffices to classify all possible $\kk[G]$-actions on $A$ up to automorphisms of $A$ and $H$ by Proposition \ref{lem:autS}. It is clear that  ${\rm Aut}_\kk(A)$ is in one-to-one correspondence with the permutation group $S_p$ of $\{e_0,\ldots,e_{p-1}\}$. One can identify $\varphi \in {\rm Aut}_\kk(A)$ with  the permutation in $S_p$, again denoted by $\varphi$, via $\varphi(e_i)=e_{\varphi(i)}$. Hence any $\kk[G]$-action on $A$ is determined by a group homomorphism $f: G\to S_p$. If  $f$ is trivial, then the $\kk[G]$-action on $A$ is trivial, so the smash product is $A \o H$. If $f$ is nontrivial, then ${\rm Im}(f)$ is a nontrivial elementary $p$-subgroup of $S_p$. Since $p^2 \nmid p!$, every nontrivial $p$-subgroup of $S_p$ is a Sylow $p$-subgroup isomorphic to $C_p$.  In particular, ${\rm Im}(f) \cong C_p$. Hence $G'=\ker(f)$ is an index $p$ subgroup of $G$, and there exists a cyclic subgroup $C \cong C_p$ of $G$ such that $G = C\oplus G'$ and $f|_C$ is injective. 
Since ${\rm Im}(f)$ is a Sylow $p$-subgroup of $S_p$, there is some automorphism $\varphi$ of $A$ (or $S_p$) such that $\varphi\circ f(g)\circ \varphi^{-1}$ sends every $e_i$ to $e_{i-1}$, where $g$ is a generator of $C$. By \eqref{action}, we can assume that $g\cdot e_i=e_{i-1}$ for all $i\in \mathbb Z_p$. Hence 
\begin{align*}
g\cdot x=\sum_{i\in\mathbb Z_p} i(g\cdot e_i)=\sum_{i\in\mathbb Z_p} (i-1)e_{i-1}+\sum_{i\in\mathbb Z_p} e_i=x+1.
    \end{align*}
    Note the above action of $g$ on $A$ coincides with the natural action $\rightharpoonup$  of  $\kk[C_p]$ on $A=\kk[C]^*$. Therefore, it gives the corresponding smash product 
   $$ 
    A\#H\cong A\# (\kk[C]\otimes \kk[G'])\cong (A\# \kk[C])\otimes \kk[G']\cong \mathcal H(\kk[C]^*)\otimes \kk[G']. 
    $$
    Finally, it is clear that these  smash products are not twist equivalent. 
\end{proof}

\section{Short exact sequences of $p$-dimensional Hopf algebras}
In the last section,  we will provide a list of the Radford algebra $R(p)$ characterizations and identify all $p^2$-dimensional Hopf algebra extensions of $p$-dimensional ones. In particular, these extensions are all pointed. 

We first recall the classification of $p$-dimensional Hopf algebras over an algebraically closed field $\kk$ of characteristic $p$. There are only two nonisomorphic $p$-restricted Lie algebras of dimension 1 generated by $x$, and they are given by (i) $\mathfrak a_1$: $x^p=x$; (ii) $\mathfrak a_2$: $x^p=0$. Note that the restricted universal enveloping algebra $u(\mathfrak a_1)=\kk[x]/(x^p-x)$ is isomorphic to the dual group algebra $\kk[C_p]^*$.
\begin{thm}\cite[Thm. 2.1]{NW19} \label{thm:p1}
   Let $\kk$ be an algebraically closed field  of characteristic $p>0$ and $H$ a $p$-dimensional Hopf algebra over $\kk$. Then $H \cong \kk[C_p]$, $u(\fa_1)$ or $u(\fa_2)$. 
\end{thm}

Independent of the classification of $p^2$-dimensional pointed Hopf algebras, we provide some characterizations of the Radford algebra $R(p)$ as the unique $p^2$-dimensional Hopf algebra over $\kk$ which satisfies any of the equivalent conditions in the following theorem.

\begin{thm}\label{thm:Radford}
     Let $\kk$ be an algebraically closed field of characteristic $p>0$, and $H$  a $p^2$-dimensional Hopf algebra over $\kk$. The following statements are equivalent:
     \begin{enumerate}[label=\rm{(\roman*)}]
         \item $H$ is isomorphic to the Radford algebra $R(p)$;
         \item  $H$ is noncommutative, pointed  and $G(H)$ is not trivial;
         \item H is noncommutative, and it admits a normal Hopf subalgebra isomorphic to $\kk[C_p]$;
         \item $H$ is a nontrivial extension, i.e., not isomorphic to the tensor Hopf algebra $\kk[C_p] \o \kk[C_p]^*$, that fits into the exact sequence
     \begin{equation} \label{eq:RadExt}
    1 \to \kk[C_p] \xrightarrow{\iota} H \xrightarrow{\pi} \kk[C_p]^* \to 1.
\end{equation}
     \end{enumerate}
     In particular, $R(p) \cong R(p)^*$.
\end{thm}
\begin{proof} It is clear that (i) implies (ii), (iii), and (iv). It remains to show the converse implications. 

\noindent (ii) $\Rightarrow$ (i) Suppose $H$ is pointed and noncommutative with nontrivial $G(H)$. Then $G(H)\cong C_p$ since $H$ is noncommutative and $\dim H =p^2$. Therefore, $\kk[G(H)] = H_0 \subsetneq H_1$ or $H$ admits a nontrivial $(1, a)$-primitive element $x$ for some $a \in G(H)$. 

Let $g$ be a generator of $G(H)$. Then $g,x$ generate a Hopf subalgebra of $H$ of dimension $>p$. Thus, $g,x$ generate $H$ as an algebra by the Nichols-Zoeller theorem. Since $\dim H = p^2$,  $P_{1,a}(H)$  is spanned by $x, a-1$ and is a $G(H)$-module by conjugation. Let $N=\kk (a-1)$, which is a trivial $G(H)$-submodule of $P_{1,a}(H)$. Note that $N$ could be zero if $a=1$ but $P_{1,g}(H)/N$ is a nonzero trivial $G(H)$-module. Thus $gxg^{-1} = x+\a (a-1)$ for some $\a \in \kk$. If $\a (a-1)=0$, then $H$ generated by $g,x$ would be commutative. Therefore, $\a (a-1) \ne 0$ which is equivalent to $\a \in \kk^\times$ and $G(H)=\langle a \rangle$. Without loss of generality, one can assume $a=g$ and replace the $(1,g)$-primitive element $x$ by $x/\a$ to obtain the relation $[g,x] = g(g-1)$ or $[x, g] = g-g^2$. In particular, $H_0$ is normal in $H$. Since $[x, g^j] = jg^{j-1}[x, g] = j(g^j-g^{j+1})$ for $0\leq j\leq p-1$, the set of eigenvalues of $\ad(x)|_{H_0}$ is $\Z_p$ and so the characteristic polynomial of $\ad(x)|_{H_0}$ is $X^p-X$. Thus, $\ad(x^p)|_{H_0}=\ad(x)^p|_{H_0}=\ad(x)|_{H_0}$. Thus, $x^p-x$ lies in the center of $H$. 
  
Let $L = H/(H^+_0 H)$ and $\pi: H \to L$ be the natural surjection. We claim that $L\cong u(\fa_1)$. It is clear that $L$ is generated by the nonzero primitive element $\pi(x)$, and so $L\cong u(\mathfrak a_i)$ for $i=1$ or $2$. If $L\cong u(\mathfrak a_2)$,  then $H$ is local by Lemma \ref{lem:SESlocal}(iii). This implies that $x^{p^2}\in (H^+)^{p^2}=0$, which contradicts to the fact that $\ad(x^{p^2})=\ad(x)^{p^2}=\ad(x)\neq 0$ on $H_0$. Therefore, $L\cong u(\mathfrak a_1)$. 
  
  Let $\ol x$ be a primitive generator for $L$ such that $\ol x^p =\ol x$. Since $\pi(x)$ is a nonzero primitive element of $L$, $\b \pi(x)=\ol x$ for some $\b \in \kk^\times$ and $\{x^j\mid j=0,\dots, p-1\}$ is a $\kk$-linearly independent subset in $H$. Define the $\kk$-linear map $\gamma: L \to H$ by $\gamma(\ol x^i) = (\b x)^i$ for $0\leq i\leq p-1$. On can check directly that $\gamma$ is a left $L$-comodule map with the convolution inverse $\ol \gamma: L \to H$ given by  $\ol \gamma(\ol x^i) = (-\b x)^i$ for $0\leq i\leq p-1$. Thus $H$ is spanned by $\{x^ig^j \mid i, j = 0,\ldots, p-1\}$ as $H$ is a left $\kk[C_p]^*$-cleft extension of $\kk[C_p]$, and hence it forms a basis for $H$ as $\dim H=p^2$.
  
  Now, we can write $x^p-x=\sum_{0\le j\le p-1} f_j(x) g^j$ for some unique polynomial $f_j(x)$ in $x$ of degree $\le p-1$.  We find
\begin{align*}
    0=[x,x^p-x]& =\sum_{0\le j\le p-1}f_j(x)[x,g] =\sum_{0\le j\le p-1}j f_j(x)(g^j-g^{j+1})\\
    & =f_1(x) g +\left(\sum_{j=1}^{p-2}
((j+1) f_{j+1}(x) -jf_j(x))g^{j+1}\right)  + f_{p-1}(x).
\end{align*}
This implies  $f_j(x)=0$ for all $j=1, \dots, p-1$. Thus $x^p-x=f_0(x)$. Applying $\pi$ to this equality, we find  
\[
(\b^{-p}-\b^{-1})\ol{x}=\b^{-p}\ol x^p-\b^{-1}\ol{x} = \pi(x^p-x)=\pi(f_0(x))=f_0(\b^{-1} \ol x)\,,
\]
or equivalently, $h(\ol x)=0$ where $h(X) = f_0(\b^{-1} X) - (\b^{-p}-\b^{-1})X \in \kk[X]$ has degree $\le p-1$. Thus, $h(X)=0$ or $f_0(X) = (\b^{(1-p)} -1)X$. Now, we find 
$$
0=[x^p-x, g]= [(\b^{(1-p)} -1) x, g] = (\b^{(1-p)} -1)g(1-g)
$$
which forces $\b^{(1-p)}-1 =0$. Hence $f_0(x)=0$ or  $x^p-x=0$ in $H$. Therefore,  $H$ is isomorphic to the Radford algebra $R(p)$. 

\noindent (iii) $\Rightarrow$ (ii) Suppose $H$ is noncommutative and admits a normal Hopf subalgebra $K \cong \kk[C_p]$.  In view of Theorem \ref{thm:basic}, it suffices to show that $H^*$ is nonsemisimple. We have an exact sequence of Hopf algebras:
$$
1 \to L^* \to H^* \to K^* \to 1,\quad \text{where }L = H/K^+ H\,.
$$
Suppose $H^*$ is semisimple. Then so is $L^*$. It follows from Theorem \ref{thm:p1} that $L^* \cong u(\fa_1)$. Since $K^*$ and $L^*$ are connected, so is $H^*$ by Lemma \ref{lem:SESlocal}(ii). However, this implies $H\cong \kk[C_p\times C_p]$ or $\kk[C_{p^2}]$ by \cite[Thm. 0.1]{Masuoka09}. This contradicts the noncommutativity of $H$. 

\noindent (iv) $\Rightarrow$ (iii)  Let $H$ be a commutative Hopf algebra which fits into the exact sequence \eqref{eq:RadExt}. It suffices to show that $H$ is a trivial extension. Note that $H^*$ satisfies the same exact sequence. Therefore, both $H$ and $H^*$ are nonsemisimple, so $H$ and $H^*$ are pointed by Theorem \ref{thm:basic}. Then $H^*$ is also commutative; otherwise, $H^* \cong R(p)$ as Hopf algebras by (ii), which would imply $H$ is noncommutative. Since $H_0 \subsetneq H_1$, there exists a nontrivial primitive element $x \in H$ and $\dim P(H) =1$. It is clear that $\kk[x]$ maps surjectively onto $\kk[C_p]^*$ and hence $\kk[x] \cong \kk[C_p]^*$ as Hopf algebras. Since $H$ is generated by $G(H)$ and $x$ as an algebra,
$H \cong \kk[G(H)]\o \kk[x] \cong \kk[C_p]\o \kk[C_p]^*$ as Hopf algebras. 

The last assertion follows directly from the equivalence of (i) and (iv).
\end{proof}

Recently, Gelaki  considered in \cite[Ex. 6.14]{Ge} the Drinfeld twist 
\begin{equation}\label{eq:twist}
J=\sum_{i=0}^{p-1} \frac{x(x-1)\cdots(x-i+1)\otimes y^i}{i!}
\end{equation}
of the restricted universal enveloping algebra  $u(\fg_5)$ of the 2-dimensional nonabelian $p$-restricted Lie algebra $\fg_5$ spanned by $x,y$ such that $[x,y]=y$, $x^p=x$ and $y^p=0$. It is pointed out that the corresponding twisted Hopf algebra $u(\fg_5)^J$ is noncommutative and noncocommutative  of dimension $p^2$ in characteristic $p$.   Now, we show that this noncommutative and noncocommutative Hopf algebra is isomorphic to the Radford algebra $R(p)$ due to the preceding theorem.  The readers are referred to \cite{KMN} for the notions of a Drinfeld twist $J$ of a Hopf algebra $H$, the corresponding twisted Hopf algebra $H^J$, and gauge equivalence of Hopf algebras.

\begin{cor}\label{GR}
Let $\kk$ be an algebraically closed field of characteristic $p>0$ and $\fg_5$ the  2-dimensional nonabelian $p$-restricted Lie algebra over $\kk$ described above. For any  Drinfeld twist $J$ of $u(\fg_5)$ such that $u(\fg_5)^J$ is noncocommutative,  $u(\fg_5)^J \cong R(p)$ as Hopf algebras. In particular, $u(\fg_5)$ and $R(p)$ are gauge equivalent. 
\end{cor}
\begin{proof} Note that $u(\fg_5)=\kk\langle x,y\rangle/(x^p-x, y^p, [x,y]-y)$ as Hopf algebras where $x, y$ are primitive. Thus $(y)=u(\fg_5)y$ is a nilpotent ideal of $u(\fg_5)$ and $u(\fg_5)/(y)  \cong \kk[x]/(x^p-x)$ is semisimple and commutative. Therefore, $(y)$ is the Jacobson radical of $u(\mathfrak g_5)$ and $u(\fg_5)^*$ is pointed with $|G(u(\fg_5)^*)|=p$. Let $J$ be a Drinfeld twist of $u(\fg_5)$ such that its associated twisted Hopf algebra $H = u(\fg_5)^J$ is noncocommutative or $H^*$ is noncommutative.  Since $H\cong u(\fg_5)$ as algebras, $H^* = u(\fg_5)^*$ as coalgebras. Therefore, $H^*$ is pointed with $|G(H^*)|=p$. Now, it follows from Theorem \ref{thm:Radford}(ii) that $H^* \cong R(p)$, and so $H \cong R(p)$ as $R(p)$ is self-dual. The last assertion follows from the existence of such a Drinfeld twist $J$ given by \eqref{eq:twist} according to \cite[Ex. 6.14]{Ge}.  
\end{proof}

\begin{remark}
The Radford algebra first appeared in the work of Radford as an example of pointed Hopf algebra of dimension $p^2$ with antipode of order $2p$ (see \cite{Rad1977}).  Taft  pointed out that it is the smash product of the commutative and cocommutative Hopf algebras $\kk[C_p]$ and $\kk[C_p]^*$, and its antipode is not  semisimple. Theorem \ref{thm:Radford} provides several characterizations of the Radford algebra. In particular, it is the unique nontrivial extension of $\kk[C_p]^*$ by $\kk[C_p]$. 
\end{remark}

Finally, we would like to classify all the extensions of $p$-dimensional Hopf algebras using the classification theorem of $p^2$-dimensional Hopf algebras. Recall that there are five nonisomorphic 2-dimensional $p$-restricted Lie algebras $\mathfrak g_i$ with basis elements $x,y$ (cf. \cite[Appendix A]{Wan}). Four of them are abelian, and they are given by
\begin{center}
 $\fg_1$: $x^p=x$, $y^p=y$;\, $\fg_2$:   $x^p=x$, $y^p=0$;\, $\fg_3$: $x^p=y$, $y^p=0$;\, $\fg_4$: $x^p=0$, $y^p=0$.
\end{center}
The $p$-restricted Lie algebra $\fg_5$ introduced in the remark preceding Corollary \ref{GR} is the unique $2$-dimensional nonabelian $p$-restricted Lie algebra up to isomorphism. 

\begin{thm}\cite[Thm. 7.4]{Wan}\cite[Thm. 3.3]{WW14} \label{thm:p2}
Let $\kk$ be an algebraically closed field  of characteristic $p>0$, and $H$  a pointed Hopf algebra of dimension $p^2$. Then $H$ is isomorphic to one of the following  14 nonisomorphic  Hopf algebras:
\begin{itemize}
\item[\rm (1)] the group algebra $\kk[G]$, where $G\cong C_{p^2}$ or $C_p\times C_p$;
\item[\rm (2)] the dual group algebra $\kk[C_{p^2}]^*$;
\item[\rm (3)] the restricted universal enveloping algebra $u(\mathfrak g_i)$ with $1\leq i\leq 5$;
\item[\rm (4)] the dual restricted universal enveloping algebra $u(\mathfrak g_i)^*$ with $i=3,5$;
\item[\rm (5)] the tensor product Hopf algebras $\kk[C_p]\otimes u(\mathfrak a_1)$ and $\kk[C_p]\otimes u(\mathfrak a_2)$;
\item[\rm (6)] the Radford algebra $R(p)$;
\item[\rm (7)] the connected and local Hopf algebra $\G_p(2):=\kk[x]/(x^{p^2})$ with $\Delta(x)=x\otimes 1+1\otimes x+\sum_{i=1}^{p-1} \frac{x^{pi}}{i!}\otimes \frac{x^{p(p-i)}}{(p-i)!}$. 
\end{itemize}
\end{thm}

\begin{remark} There are relationships between these Hopf algebras listed in Theorem \ref{thm:p2}.
\begin{enumerate}  [label=\rm{(\roman*)}]
\item  We have the isomorphisms of Hopf algebras  $u(\mathfrak g_1)^*\cong \kk[C_p\times C_p]$ and $u(\mathfrak g_2)^*\cong \kk[C_p]\otimes u(\mathfrak a_2)$.
\item The Hopf algebras $u(\mathfrak g_4)$, $\kk[C_p]\otimes u(\fa_1)$, $R(p)$ and $\mathcal G_p(2)$ are self-dual.
\item  The Radford algebra $R(p)$ has a $\kk$-basis $\{x^ig^j\}_{0\leq i,j\leq p-1}$. Let $\gamma, \chi \in R(p)^*$  defined by $\gamma(x^ig^j)=(-1)^i$, $\chi(x^ig^j)=-j(-1)^i$. One can check directly that $\gamma$ is a character of $R(p)$, and $\chi$ is $(\epsilon,\gamma)$-primitive which satisfy
 $\gamma^p=1, \chi^p=\chi$ and $[\gamma,\chi]=\gamma(\gamma-1)$. The assignment $g \mapsto \gamma, x \mapsto \chi$ extends to a Hopf algebra isomorphism from $R(p)$ to $R(p)^*$.
\item  Let $\{\delta_i\}_{i=0}^{p^2-1}$ denote the dual basis of $\{x_i\}_{i=0}^{p^2-1}$ for $(\kk[x]/(x^{p^2}))^*$. 
One can check directly that $(\delta_p)^p = \delta_1$ and $\Delta(\delta_{p})=\delta_{p}\otimes 1+1\otimes \delta_{p}+\sum_{i=1}^{p-1} \frac{(\delta_1)^i}{i!}\otimes \frac{(\delta_1)^{p-i}}{(p-i)^!}$. Therefore, 
 $\mathcal G_2(p)^*=\kk[\delta_{p}]$ and the map $x \mapsto \delta_p$ extends to a Hopf algebra isomorphism from $\G_p(2)$ and $\G_p(2)^*$. Note that $\mathcal G_p(2)$ is not primitively generated, and its comultiplication of $x$  appears in the  2-step iterated Hopf-Ore extensions (IHOEs) in positive characteristic \cite{BZ20}. 

\item[(v)] The Hopf algebras $R(p)$ and $u(\fg_5)$ are gauge equivalent by Corollary \ref{GR}.
\end{enumerate}
 \end{remark}
The following corollary is an immediate consequence of the above remark.
\begin{cor}\label{c:dual_closure}
    The pointed Hopf algebras over $\kk$ of dimension $p^2$ are closed under duality.
\end{cor}
Now, we can provide an affirmative answer to Question B for those $p^2$-dimensional Hopf algebras which are extensions of $p$-dimensional ones. 

\begin{thm}\label{cor:extp2}
Let $H$ be a $p^2$-dimensional Hopf algebra over an algebraically closed field $\kk$ of characteristic $p>0$. Then $H$  admits a normal Hopf subalgebra of dimension $p$ if and only if $H$ is pointed. In particular, every pointed $p^2$-dimensional Hopf algebra is an extension of $p$-dimensional ones.
\end{thm}
\begin{proof}
If $H$ is pointed, it follows from the classification Theorem \ref{thm:p2} that $H$ has  a normal Hopf subalgebra of dimension $p$. Conversely, suppose $H$ has a normal Hopf subalgebra $K$ of dimension $p$. It suffices to show that $H$ or $H^*$ is pointed by Corollary \ref{c:dual_closure}. Now, $H$ fits into a short exact sequence of Hopf algebras
\[
1\to K\to H\to L\to 1,
\]
where $L=H/K^+H$ is another $p$-dimensional Hopf algebra by the Nichols-Zoeller theorem. By Theorem \ref{thm:p1}, $L, K \cong \kk[C_p]$, $u(\fa_1)$ or $u(\fa_2)$. If $L \cong u(\fa_i)$, then $L$ is connected and so $H$ is pointed by Lemma \ref{lem:SESlocal}(ii). On the other hand, if $K \cong  \kk[C_p]$ or $u(\fa_2)$, then $K$ is local and so $H^*$ is pointed by  Lemma \ref{lem:SESlocal}(i). The remaining case is that $K \cong u(\fa_1)$ and $L \cong \kk[C_p]$. In particular, $H$ is an $L$-cleft extension of $K$. It follows from Proposition \ref{ExtG} that $H \cong K \o L$ as $\kk$-algebras, which is commutative. Thus, $H^*$ is pointed.
\end{proof}

As a conclusion of our section, we list all possible nontrivial exact sequences of Hopf algebras 
\[
1\to K\to H\to L\to 1,
\]
where $K$ and $L$ are $p$-dimensional Hopf algebras over $\kk$. 
\begin{table}[H]\label{T2}
\caption{Nontrivial exact sequences $1\to K\to H\to L\to 1$}
\begin{center}
\begin{tabular}{||c|c | c ||} 
 \hline\hline
$K$ & $H$ & $L$ \\
\hline
$\kk[C_p]$ & $\kk[C_{p^2}]$ & $\kk[C_p]$ \\
\hline
$\kk[C_p]$ &  $R(p)$ & $u(\mathfrak a_1)$\\
\hline
$\kk[C_p]$ & $u(\mathfrak g_5)^*$ & $u(\mathfrak a_2)$\\
\hline
$u(\mathfrak a_1)$ & None  & $\kk[C_p]$\\
\hline
$u(\mathfrak a_1)$ &  $\kk[C_{p^2}]^*$ & $u(\mathfrak a_1)$\\
\hline
$u(\mathfrak a_1)$ &  None & $u(\mathfrak a_2)$\\
\hline
$u(\mathfrak a_2)$ &  None & $\kk[C_p]$\\
\hline
$u(\mathfrak a_2)$ & $u(\mathfrak g_5)$ & $u(\mathfrak a_1)$\\
\hline
$u(\mathfrak a_2)$ & $u(\mathfrak g_3)$, $u(\mathfrak g_3)^*$, $\mathcal G_p(2)$ & $u(\mathfrak a_2)$\\
\hline
\hline
\end{tabular}
\end{center}
\end{table}

\begin{remark}
    The absence of a nontrivial extension of $\kk[C_p]$ by $\kk[C_p]^*$ can be proven without using Theorem \ref{thm:p2}. If $H$ is an extension, then $H^*$ can also fit into that same extension.  By Proposition \ref{ExtG}, both $H$ and $H^*$ are the tensor product algebra $\kk[C_p]^*\otimes \kk[C_p]$. In particular, $H$ admits a nonzero primitive element $x$ and a group-like element $g$ of order $p$. Since $H$ is generated by the commuting elements $g,x$,  $H \cong  \kk[x]\o \kk[C_p]$ as Hopf algebras.  Since $\dim P(H) = 1$ and $\kk[C_p]^*$ admits a nonzero primitive element $x \in \kk[C_p]^*$, we have $\kk[x]\cong  \kk[C_p]^*$.
    \end{remark}

\end{document}